\documentclass[11pt,reqno]{amsart}

\usepackage{amsmath}
\usepackage{amsthm}
\usepackage{graphicx}
\usepackage{mathtools}
\usepackage{amsfonts}
\usepackage{amssymb}
\usepackage{hyperref}
\usepackage{bm}
\usepackage[margin=1 in]{geometry}
\usepackage{enumerate}
\usepackage[normalem]{ulem}
\usepackage{tikz}
\usepackage{xcolor}
\usetikzlibrary{matrix,arrows,positioning,automata}
  \usepackage{cancel}

\numberwithin{equation}{section}

\newtheorem{theorem}{Theorem}[section]
\newtheorem{lemma}[theorem]{Lemma}
\newtheorem{proposition}[theorem]{Proposition}
\newtheorem{corollary}[theorem]{Corollary}
\newtheorem{assumption}[theorem]{Assumption}

\theoremstyle{definition}
\newtheorem{definition}[theorem]{Definition}
\newtheorem{remark}[theorem]{Remark}

\def\E{{\mathbb E}}

\def\R{{\mathbb R}}
\def\N{{\mathbb N}}
\def\PP{{\mathbb P}}

\def\FF{{\mathbb F}}

\def\P{{\mathcal P}}

\def\QQ{{\mathbb Q}}

\def\W{{\mathcal W}}

\def\F{{\mathcal F}}
\def\tr{{\mathrm{Tr}}}

\def\Var{{\mathrm{Var}}}

\def\Pprod{\P_{\mathrm{pr}}}

\def\Punif{\P_{\mathrm{Unif}}([0,1] \times \R)}

\makeatletter
\def\@tocline#1#2#3#4#5#6#7{\relax
	\ifnum #1>\c@tocdepth 
	\else
	\par \addpenalty\@secpenalty\addvspace{#2}%
	\begingroup \hyphenpenalty\@M
	\@ifempty{#4}{%
		\@tempdima\csname r@tocindent\number#1\endcsname\relax
	}{%
		\@tempdima#4\relax
	}%
	\parindent\z@ \leftskip#3\relax \advance\leftskip\@tempdima\relax
	\rightskip\@pnumwidth plus4em \parfillskip-\@pnumwidth
	#5\leavevmode\hskip-\@tempdima
	\ifcase #1
	\or\or \hskip 1em \or \hskip 2em \else \hskip 3em \fi%
	#6\nobreak\relax
	\dotfill\hbox to\@pnumwidth{\@tocpagenum{#7}}\par
	\nobreak
	\endgroup
	\fi}
\makeatother

\title{Mean field approximations via log-concavity}
\author{Daniel Lacker, Sumit Mukherjee, and Lane Chun Yeung} 
\thanks{D.L.\ and L.C.Y.\ are partially supported by the AFOSR Grant FA9550-19-1-0291 and the NSF CAREER award DMS-2045328. S.M.\ is partially supported by NSF grants DMS-1712037 and DMS-2113414.}
\address{Department of Industrial Engineering \& Operations Research, Columbia University}
\email{daniel.lacker@columbia.edu, l.yeung@columbia.edu}
\address{Department of Statistics, Columbia University}
\email{sm3949@columbia.edu}

\begin{document}
	\begin{abstract}
	We propose a new approach to deriving quantitative mean field approximations for any probability measure $P$ on $\R^n$ with density proportional to $e^{f(x)}$, for $f$ strongly  concave. We bound the mean field approximation for the log partition function $\log \int e^{f(x)}dx$ in terms of $\sum_{i \neq j}\E_{Q^*}|\partial_{ij}f|^2$, for a semi-explicit probability measure $Q^*$ characterized as the unique mean field optimizer, or equivalently as the minimizer of the relative entropy $H(\cdot\,|\,P)$ over product measures. This notably does not involve metric-entropy or gradient-complexity concepts which are common in prior work on nonlinear large deviations. Three implications are discussed, in the contexts of continuous Gibbs measures on large graphs, high-dimensional Bayesian linear regression, and the construction of decentralized near-optimizers in high-dimensional stochastic control problems. Our arguments are based primarily on functional inequalities and the notion of displacement convexity from optimal transport.
\end{abstract}

\maketitle

\section{Introduction} \label{sec: setup}

  At the center of the recent theory of \emph{nonlinear large deviations} is the problem of justifying the \emph{mean field approximation} for the partition function of a Gibbs measure. Given a (reference) probability measure $\mu$ on $\R$, suppose a probability measure $P$ on $\R^n$ takes the form
\begin{align*}
	 P(dx) = Z^{-1} e^{f(x)}\mu^{\otimes n}(dx),
\end{align*}
for a function $f : \R^n \to \R$ and normalizing constant $Z$, where $\mu^{\otimes n}$ denotes the $n$-fold product measure. A recurring problem in diverse applications is the approximation of the often intractable \emph{partition function} $Z$. It obeys the well-known Gibbs variational principle 
\begin{align}
	\log Z = \log \int_{\R^n} e^f\,d\mu^{\otimes n} = \sup_{Q \in \P(\R^n)} \left( \int_{\R^n} f\,dQ - H(Q\,|\,\mu^{\otimes n})\right), \label{Gibbs}
\end{align}
where $\P(\R^n)$ is the set of probability measures on $\R^n$, and $H$ denotes the relative entropy
\begin{align*}
	H(Q\,|\,Q') \coloneqq \int_{\R^n} \frac{dQ}{dQ'}\log\frac{dQ}{dQ'}\,dQ' \  \text{ if } Q \ll Q', \qquad H(Q\,|\,Q') \coloneqq \infty \ \text{ if } Q \not\ll Q'.
\end{align*}
Note that $Q=P$ is the unique optimizer in \eqref{Gibbs}.
Letting $\Pprod(\R^n)$ denote the set of product measures $Q = Q_1 \times \cdots \times Q_n$ in $\P(\R^n)$, the \emph{mean field approximation} is
\begin{align}
	\log \int_{\R^n} e^f\,d\mu^{\otimes n} \approx \sup_{Q \in \Pprod(\R^n)} \left( \int_{\R^n} f\,dQ - H(Q\,|\,\mu^{\otimes n})\right). \label{MFapprox-heuristic}
\end{align}
In the cases studied in this paper, the left-hand side is expected to be of order $n$; a precise formulation of \eqref{MFapprox-heuristic} is then to find conditions under which the difference is $o(n)$, so that the mean field approximation becomes asymptotically correct at the leading order. Note that the right-hand side of  \eqref{MFapprox-heuristic} is trivially a lower bound for the left, because of \eqref{Gibbs}, and it is only the upper bound that incurs an error which must be estimated.

The groundbreaking work of \cite{chatterjee-dembo}, motivated by applications to subgraph counts in sparse random graphs, showed how to justify the mean field approximation in the case that $\mu^{\otimes n}$ is the uniform measure on the hypercube $\{-1,1\}^n$. Their key assumption is that the gradient of $f$ has \emph{low complexity}, as measured by the metric entropy of the range $\nabla f(\{-1,1\}^n)$. A number of subsequent papers have since refined this approach and results on subgraph counts \cite{cook-dembo,cook2021regularity,lubetzky2017variational}, in addition to other noteworthy applications such as Ising models \cite{augeri1,augeri2,basak-mukherjee,deb-mukherjee,eldan1,jain2019mean}.
Most applications thus far involve discrete $\mu$, but the theory has been extended to compactly supported measures \cite{augeri1,austin2019structure,yan2020nonlinear}.
Alternative and often more convenient estimates have appeared, still based on ``gradient complexity" but quantifying it in a different way, eschewing covering number estimates in favor of the simpler and weaker Gaussian-width \cite{eldan1,eldan2020taming,eldan-gross} or Rademacher-width \cite{augeri2}. 

	In this paper, we propose an alternative approach to the mean field approximation, designed most notably for the case where $f$ is concave and the reference measure $\mu$ is strongly log-concave (see Theorem \ref{th:main-intro} and Corollary \ref{co:refmeas}).
	In particular, we deal with continuous $\mu$ of unbounded support, which covers a rather different host of applications, discussed in Section \ref{se:applications}, compared to the somewhat more discrete-oriented prior literature.
	Our approach is based on a semi-explicit representation for the mean field optimizer $Q^*$ in \eqref{MFapprox-heuristic}, which we show to be unique as soon as $P$ is strictly log-concave, and which is in fact also the unique minimizer of $H(\cdot\,|\,P)$ over product measures.
	We control the error in the approximation \eqref{MFapprox-heuristic} by a constant times $\E_{Q^*}\sum_{i \neq j} |\partial_{ij}f|^2$, which is typically much simpler to work with compared to the aforementioned notions of gradient complexity.
	Eldan \cite{eldan1,eldan2020taming} and Austin \cite{austin2019structure} also analyze the mean field approximation by  approximating $P$ by product measures in entropy, but our methods and bounds are very different from theirs; notably, they approximate $P$ not by a single product measure but by a mixture, which is natural when the mean field optimizer is not unique, as is explained well in \cite{eldan1}. The uniqueness of  the mean field optimizer in our  setting means that we expect $P$ to concentrate around a single \emph{pure state}, rather than a mixture of states. 

	In the rest of this section, we describe our general results on mean field approximations for log-concave measures, 
	along with some related ideas and generalizations,
	with proofs deferred to Section \ref{se:proof-mainthm}.
 Section \ref{se:applications} develops three applications: Gibbs measures with heterogeneous interactions, high-dimensional Bayesian linear regression, and high-dimensional stochastic control problems.

\subsection{Main results}

Recall for $\kappa \in \R$ that a function $f : \R^n \to \R \cup \{-\infty\}$ is said to be \emph{$\kappa$-concave} if $x\mapsto f(x) + \frac{\kappa}{2}|x|^2$ is concave.
If $f$ is finite-valued and $C^2$, i.e., twice continuously differentiable, then $f$ is $\kappa$-concave if and only if $\nabla^2f(x) \le -\kappa I$ in semidefinite order, for each $x \in \R^n$. We say that a probability measure $P$ on $\R^n$ is $\kappa$-log-concave if it takes the form $P(dx)=e^{f(x)}dx$ for some $\kappa$-concave function $f$.
We will work with the (negative of the) differential entropy
\begin{equation*}
	H(Q) := \int_{\R^n} Q(x) \log Q(x) \, dx, \label{eq: diff.entropy}
\end{equation*}
for an absolutely continuous probability measure $Q(dx)=Q(x)dx$ on a Euclidean space, well-defined in $(-\infty,\infty]$ whenever the negative part of $Q\log Q$ is integrable; we adopt the convention that $H(Q)=\infty$ if $Q$ is not absolutely continuous, or if $(Q\log Q)^-$ is not integrable. 
Let $\Pprod(\R^n)$ denote the set of product measures on $\R^n$. Let $X =(X_1,\ldots,X_n) : \R^n \to \R^n$ denote the identity map, so that we may write $\E_Q[g(X)]=\int_{\R^n} g\,dQ$ for the expectation under $Q$.

\begin{theorem} \label{th:main-intro}
	Consider a $C^2$ and $\kappa$-log-concave probability measure $P(dx)=Z^{-1}e^{f(x)}dx$,  for some $ \kappa > 0 $. Assume there exist $ c_1 \ge 0 $ and $0 \le c_2  < \kappa/2$ such that $ |f(x)| \le  c_1 e^{c_2 |x|^2} $ for all $ x \in \R^n $. Then the following conclusions hold:
	\begin{enumerate}[(1)]
		\item \label{main.fp} There exists a unique product measure $Q^* = Q^*_1 \times \cdots \times Q^*_n \in \Pprod(\R^n)$  with strictly positive density a.e.\ satisfying  $f \in L^1(Q^*)$ and the fixed point equation
		\begin{align}
			Q^*_i(dx_i) = Z_i^{-1} \exp \big(\E_{Q^*}[f(X)\,|\,X_i=x_i]\big)\,dx_i, \qquad  Z_i > 0, \ i=1,\ldots,n.  \label{fixedpoint-mainthm}
		\end{align}
		\item $Q^*$ is $\kappa$-log-concave. \label{Q^*-log-concave}
		\item \label{main.unique}$Q^*$ is the unique optimizer in
		\begin{equation}
			\sup_{Q \in \Pprod(\R^n)} \left( \int_{\R^n} f\,dQ- H(Q)\right).  \label{MFopt-mainthm}
		\end{equation}
		\item \label{main.bound}If we define
		\begin{equation*}
			R_f  := \log \int_{\R^n}  e^{f(x)}\,dx -  \sup_{Q \in \Pprod(\R^n) } \left( \int_{\R^n} f\,dQ - H(Q)\right), 
		\end{equation*}
		then
		\begin{equation}
			0 \le R_f \le  \frac{1}{2\kappa} \E_{Q^*}\sum_{i=1}^n  \Var_{Q^*}(\partial_i f(X)\,|\,X_i)  \le \frac{1}{\kappa^2}\sum_{1\le i<j\le n} \E_{Q^*}[|\partial_{ij}f(X)|^2].  \label{MFapprox-mainthm}
		\end{equation}
	\end{enumerate}
\end{theorem}


The supremum in \eqref{MFopt-mainthm} is finite, as we will see in Lemma \ref{le:MFsup-finite}. Also, as will be seen in the proof of Proposition \ref{pr:fp-to-optimizer}, our assumptions ensure that $ \partial_if (x_i, \cdot) \in L^1\big(\prod_{j \neq i}Q^*_j\big)$, so the conditional variance in \eqref{MFapprox-mainthm} is well-defined in $ [0, \infty] $. The final quantity in \eqref{MFapprox-mainthm} controlling our mean field approximation error involves only the cross-derivatives $i \neq j$, which are insensitive to additively separable perturbations  $f(x) \to f(x) + \sum_{i=1}^n\tilde{f}_i(x_i)$. On the other hand, the measure $Q^*$ is sensitive to these perturbations, but in the tractable sense that $Q^*_i(dx_i)$ must be multiplied by $\exp \tilde{f}_i(x_i)$ (and a new normalizing constant). In particular,  both upper bounds in \eqref{MFapprox-mainthm} vanish if $f$ is already additively separable, i.e., if $P$ is a product measure.

In Theorem \ref{th:main-intro}, the measure $Q^*$ is defined implicitly, which can make bounding $R_f$ difficult. 
In the simplest case where $\nabla^2f$ is bounded, we need no knowledge of $Q^*$ to obtain
\begin{align*}
	R_f \le \frac{1}{\kappa^2}\sup_{x \in \R^n}\sum_{1 \le i < j \le n} |\partial_{ij}f(x)|^2,
\end{align*}
which is sharp enough for many applications.
But even when $\nabla^2f$ is unbounded, we can take advantage of the fact that $Q^*$ is $\kappa$-log-concave by Theorem \ref{th:main-intro}(2), which implies in particular that it has finite moments of all orders controlled in terms of $\kappa$.

	A guiding example is the class of Gibbs measures with pairwise interactions of the form
	\begin{align}
		f(x) &= \sum_{i=1}^n V(x_i) + \sum_{1 \le i < j \le n} J_{ij} K(x_i-x_j), \label{def:Gibbs-potential}
	\end{align}
	where $V$ is $\kappa$-concave, $K$ is even and concave, and $J$ is a symmetric matrix with nonnegative entries. Then $\partial_{ij}f(x)=-J_{ij}K''(x_i-x_j)$ for $i \neq j$, and for $K''$ bounded we immediately deduce $R_f \le \tr(J^2)\|K''\|_\infty^2/2\kappa^2$ from Theorem \ref{th:main-intro}.
	Corollary \ref{cor:quadratic} below proves a similar $O(\tr(J^2))$ bound merely assuming that $K''$ has at most exponential growth, plus a symmetry assumption.
	Since $\log\int_{\R^n} e^f dx$ is order $n$ in this case, we obtain a successful mean field approximation whenever $J$ satisfies $\tr(J^2)=o(n)$, a well-established condition in the literature. We postpone to Section \ref{se:Gibbs} further discussion of this class of examples.

As a first corollary of Theorem \ref{th:main-intro}, we deduce the following non-asymptotic law of large numbers for the empirical measure.

\begin{corollary} \label{co:linearstat-concentration}
	Under the assumptions of Theorem \ref{th:main-intro}, for any 1-Lipschitz function $\varphi : \R \to \R$, we have
	\begin{equation}
		\E_P\bigg[ \bigg(\frac{1}{n}\sum_{i=1}^n\varphi(X_i) - \frac{1}{n}\sum_{i=1}^n\E_{Q^*}[\varphi(X_i)]\bigg)^2\bigg]  \le \frac{(1 +  \sqrt{2R_f})^2}{\kappa n} . \label{ineq:weakLLN}
	\end{equation}
\end{corollary}

\begin{remark} 
	Corollary \ref{co:linearstat-concentration} can be interpreted as a form of concentration of the empirical measure $\frac{1}{n}\sum_{i=1}^n \delta_{X_i}$ around the measure $\frac{1}{n}\sum_{i=1}^n Q^*_i$. Alternatively, the Poincar\'e inequality for $P$ implies $\Var_P(\frac{1}{n}\sum_{i=1}^n\varphi(X_i)) \le 1/\kappa n$ for 1-Lipschitz $\varphi$, which in turn implies a form of concentration of $\frac{1}{n}\sum_{i=1}^n \delta_{X_i}$ around its mean $\frac{1}{n}\sum_{i=1}^n P_i$, where $P_i$ is the $i^\text{th}$ marginal of $P$. However, the latter is normally not as useful, because the marginals of $P$ are typically not as tractable as the various characterizations of $Q^*$ provided by Theorem \ref{th:main-intro}.
\end{remark}

It is often convenient to work with a probability measure as a reference measure, in place of Lebesgue measure, as is common in the literature on mean field approximations (see for example \cite{augeri1,austin2019structure,chatterjee-dembo,eldan1,yan2020nonlinear}). Theorem \ref{th:main-intro} implies a similar result in terms of reference probability measures.

\begin{corollary} \label{co:refmeas}
	Let $V_i : \R \to \R$ be $C^2$ and $\kappa$-concave  for some $\kappa > 0$, such that $\rho_i(dx)=e^{V_i(x)}dx$ is a probability measure, for $i=1,\ldots,n$. Let $\rho = \rho_1 \times \cdots \times \rho_n$.
	Let $g: \R^n \to \R$ be  $C^2$ and concave. Assume there exist $ c_1\ge 0 $ and $0 \le c_2  <  \kappa /2$ such that $ |g(x)| \le c_1 e^{c_2 |x|^2} $ for all $ x \in \R^n $.  Then the following conclusions hold: 
	\begin{enumerate}[(1)]
		\item There exists a unique product measure $Q^* = Q^*_1 \times \cdots \times Q^*_n \in \Pprod(\R^n)$  with strictly positive density a.e.\ satisfying  $g \in L^1(Q^*)$ and 
		\begin{align}
			Q^*_i(dx_i) = Z_i^{-1} \exp \big(\E_{Q^*}[g(X)\,|\,X_i=x_i]\big)\,\rho_i(dx_i), \qquad  Z_i > 0, \ i=1,\ldots,n.  \label{fixedpoint-refmeas}
		\end{align}

		\item $Q^*$ is $\kappa$-log-concave. \label{Q^*-log-concave-refmeas}

		\item $Q^*$ is the unique optimizer in
		\begin{equation}
			\sup_{Q \in \Pprod(\R^n)} \left( \int_{\R^n} g\,dQ- H(Q\,|\,\rho )\right).  \label{MFopt-refmeas}
		\end{equation}
		\item 
		If we define
		\begin{equation*}
			R^\rho_g  := \log \int_{\R^n}  e^g\,\,d\rho -  \sup_{Q \in \Pprod(\R^n) } \left( \int_{\R^n} g\,dQ - H(Q\,|\,\rho )\right),
		\end{equation*}
		then
		\begin{equation}
			0 \le R^\rho_g \le  \frac{1}{2\kappa} \E_{Q^*}\sum_{i=1}^n  \Var_{Q^*}(\partial_i g(X)\,|\,X_i)  \le \frac{1}{\kappa^2}\sum_{1\le i<j\le n} \E_{Q^*}[|\partial_{ij}g(X)|^2]. \label{MFapprox-refmeas}
		\end{equation}
	\end{enumerate}
\end{corollary}

 For certain symmetric choices of $g$, the bound \eqref{MFapprox-refmeas} is related to the theorems of Cram\'er and Sanov on large deviations, which are settings in which the Gibbs variational principle is well known to be nearly saturated by product measures. 
	For instance, if $g(x) = nG\big(\frac{1}{n}\sum_{k=1}^n x_k \big)$ for some continuous concave $G$, we obtain $R^\rho_g \le \|G''\|_\infty^2/2\kappa^2$, which is certainly $o(n)$ when $G''$ is bounded.

\subsection{Overview and proof ideas} \label{se:proofideas}
We explain here some key ideas behind Theorem \ref{th:main-intro} and its corollaries.
The simple identity 
\begin{align}
	\log \int_{\R^n} e^{f(x)}\,dx - \int_{\R^n} f\,dQ + H(Q) = H(Q\,|\,P) \label{entproj-id0}
\end{align}
is valid for probability measures $Q$ with finite entropy and implies (see Lemma \ref{le:MFsup-finite} for details)
\begin{align}
	\log \int_{\R^n} e^{f(x)}\,dx - \sup_{Q \in \Pprod(\R^n)}\left(\int_{\R^n} f\,dQ - H(Q)\right) = \inf_{Q \in \Pprod(\R^n)} H(Q\,|\,P), \label{entproj-id1}
\end{align}
and also that optimizing \eqref{MFopt-mainthm} is equivalent to optimizing
\begin{align}
	\inf_{Q \in \Pprod(\R^n)} H(Q\,|\,P).  \label{entproj}
\end{align}
That is, $Q^*$ from Theorem \ref{th:main-intro} is the optimizer in \eqref{entproj}.
This can be seen as an \emph{entropic projection}, in the sense of Csiszar \cite{csiszar1975divergence}, onto the set of product measures. A minimizer in \eqref{entproj} always exists, because the set of product measures is weakly closed and $H(\cdot\,|\,P)$ has weakly compact sub-level sets. But uniqueness is not obvious and in fact fails in general, because the set of product measures is not convex.
We establish the uniqueness of the optimizer in Lemma \ref{lem: maximizer unique} in the case where $P$ is strictly log-concave, by exploiting the notion of displacement convexity from the theory of optimal transport, with similarities to the work of McCann \cite{mccann1997convexity}.

Once we know that the optimizer $Q^*$ for \eqref{MFopt-mainthm} takes the form \eqref{fixedpoint-mainthm}, the proof of the mean field approximation \eqref{MFapprox-mainthm} is fairly quick, if we ignore certain technical points:
The right-hand side of the identity \eqref{entproj-id1} is precisely $H(Q^*\,|\,P)$.  We first use the log-Sobolev inequality for $P$, which is ensured by $\kappa$-log-concavity and the famous result of Bakry-\'Emery \cite{bakryemery}, to get
\begin{align*}
	H(Q^*\,|\,P) &\le \frac{1}{2\kappa}\int_{\R^n}\left|\nabla \log \frac{dQ^*}{dP}\right|^2\,dQ^*.
\end{align*}
Since $Q^*=Q^*_1\times \cdots \times Q^*_n$ is a product measure,  the formula \eqref{fixedpoint-mainthm} implies
\begin{align}
	\partial_i \log Q^*(x) = \partial_i\log Q^*_i(x_i) = \partial_i\E_{Q^*}[ f(X)\,|\,X_i=x_i] = \E_{Q^*}[\partial_i f(X)\,|\,X_i=x_i]. \label{eq:diffQ*}
\end{align}
Thus,
\begin{equation*}
	H(Q^*\,|\,P) \le \frac{1}{2\kappa} \E_{Q^*}\sum_{i=1}^n \left(\E_{Q^*}[\partial_i f(X)\,|\,X_i] - \partial_if(X)\right)^2  = \frac{1}{2\kappa}\E_{Q^*}\sum_{i=1}^n \Var_{Q^*}(\partial_i f(X)\,|\,X_i).
\end{equation*}
Differentiating \eqref{eq:diffQ*} again shows easily that $Q^*$ is $\kappa$-log-concave since $f$ is concave. Hence, $Q^*$ and its marginals obey a Poincar\'e inequality, and we deduce
\begin{align*}
	\Var_{Q^*}(\partial_i f(X)\,|\,X_i) \le \frac{1}{\kappa}\sum_{j \neq i}\E_{Q^*}\left[|\partial_{ij}f(X)|^2\,|\,X_i\right].
\end{align*}
Combining the last two inequalities yields \eqref{MFapprox-mainthm}. See Section \ref{sec: generalization} below for a discussion of a generalization of this argument beyond the strongly log-concave case.

The proof of Corollary \ref{co:linearstat-concentration} begins with the observation that the $\kappa$-log-concavity of $P$ in Theorem \ref{th:main-intro} implies the quadratic transport inequality \cite[Theorems 1 and 2]{otto2000generalization}
\begin{align}
	\W_2^2(Q^*,P) \le \frac{2}{\kappa}H(Q^*\,|\,P), \label{T2I}
\end{align}
where $\W_2$ denotes the quadratic Wasserstein distance defined by
\begin{align*}
	\W_2^2(Q^*,P) = \inf_\pi \int_{\R^n \times \R^n} |x-y| ^2\,\pi(dx,dy),
\end{align*}
where the infimum is over $\pi \in \P(\R^n \times \R^n)$ with marginals $Q^*$ and $P$. Combining \eqref{T2I} with the inequality $H(Q^*\,|\,P) \le R_f$ discussed above, we arrive at $\W_2^2(Q^*,P) \le 2R_f/\kappa$. 
The quadratic Wasserstein distance enjoys a useful and fairly well known \emph{subadditivity} inequality, which we prove in Section \ref{subsec: proofs-asymp-indep} for the sake of completeness: If 
$P_S$ denotes the marginal law of $(X_i)_{i \in S}$ under $P$ for a set $S \subset [n]:=\{1,\ldots,n\}$, and similarly for $Q^*_S$, then we have
\begin{align}
	\binom{n}{k}^{-1}\sum_{S \subset [n], \, |S|=k} \W_2^2(Q^*_S,P_S) \le \frac{1}{\lfloor n/k\rfloor}\W_2^2(Q^*,P) \le \frac{2 }{\kappa\lfloor n/k\rfloor}R_f  \le \frac{4k}{n\kappa} R_f \label{W2subadditivity}
\end{align}
for any  $1 \le k \le n$. 
With \eqref{W2subadditivity} in hand, the proof of Corollary \ref{co:linearstat-concentration} is straightforward.
	Moreover, in our cases of interest where $R_f = o(n)$, the bound \eqref{W2subadditivity} quantifies a form of \emph{approximate independence}: Most $k$-particle marginals of $P$ are $\W_2$-close to product measures, if $k=o(n/R_f)$.

	\begin{remark}
		We work throughout the paper with state space $\R$, for simplicity. That is, we study approximations of measures on $\R^n$ by $n$-fold products of measures on $\R$, as opposed to, say, approximations of measures on $(\R^d)^n$ by $n$-fold products of measures on $\R^d$. Most of our arguments, based primarily on convexity and functional inequalities, extend to the case of $\R^d$ or even Riemannian manifolds with lower curvature bounds in the spirit of Bakry-\'Emery \cite{bakryemery,bakry2013analysis}. The only difficulty is in the uniqueness claimed in Theorem \ref{th:main-intro} (proven in Proposition \ref{pr:fp-to-optimizer}), which would require a finer analysis involving regularity of certain optimal transport maps.
	\end{remark}

\subsection{Additional discussion and results}

	The remaining results presented in this section will not be used in the rest of the paper but serve to elaborate on the structure of the main theorem. The reader mainly interested in applications or proofs of the above results may skip to Sections \ref{se:applications} or \ref{se:proof-mainthm}, respectively, with no loss of continuity.

\subsubsection{More on entropic projections}
Reversing the order of arguments in the relative entropy in \eqref{entproj} leads to a very different optimization problem, but  it is instructive to compare the two. 
The infimum
\begin{equation}
	\inf_{Q \in \Pprod(\R^n)}H(P\,|\,Q) \label{entproj-reversed}
\end{equation}
is uniquely attained by taking $Q=P^*:=P_1\times \cdots \times P_n$ to be the product of the marginals of $P$. 
Indeed, from the simple identity $H(P\,|\,Q) = H(P\,|\,P^*) + H(P^*\,|\,Q)$, it follows that $H(P\,|\,Q) \ge H(P\,|\,P^*)$ for all $Q$, with equality if any only if $Q=P^*$. 

The Gaussian case highlights the difference between \eqref{entproj-reversed} and \eqref{entproj}. Suppose $P$ is a centered Gaussian with nonsingular covariance matrix $\Sigma$. In this case it is easy to see that the (unique) minimizer of $H(Q\,|\,P)$ among product measures $Q$ is the centered Gaussian with covariance matrix $\widetilde\Sigma$, where $\widetilde\Sigma^{-1}$ is the diagonal matrix obtained by deleting the off-diagonal entries of $\Sigma^{-1}$. On the other hand, the unique minimizer of $H(P\,|\,Q)$ among product measures $Q$ is the centered Gaussian with covariance matrix $\widehat\Sigma$ obtained by deleting the off-diagonal entries of $\Sigma$.

\subsubsection{Tilts}
A similar bound to Corollary \ref{co:refmeas} is available if one seeks a stronger mean field approximation, in which $\Pprod(\R^n)$ is replaced by the sub-class of product measures given by \emph{tilts} of a given reference measure.
We focus on the case of Gaussian reference measure, as it is not obvious how to extend the argument to a general reference measure. For $y \in \R^n$, let $\gamma_{y,t}$ denote the Gaussian with mean $y$ and covariance matrix $tI$, with $\gamma_t:=\gamma_{0,t}$, noting that $\gamma_{y,t} \in \Pprod(\R^n)$.

\begin{proposition} \label{pr:tilt-Gaussian}
	Let $t > 0$, and let $f : \R^n \to \R$ be $C^2$ and concave. Assume there exist $c_1 \ge 0$ and $0 \le c_2 < 1/2t$ such that $|f(x)| \le c_1e^{c_2|x|^2}$. Then there is a unique $y^* \in \R^n$ satisfying
	\begin{align}
		y^* = t \int_{\R^n} \nabla f \, d\gamma_{y^*,t}, \label{Gaussian-tilt-eq}
	\end{align}
	and it holds that
	\begin{align}
		\log\int_{\R^n} e^f\,d\gamma_t \le \sup_{y \in \R^n}\left( \int_{\R^n} f\,d\gamma_{y,t} - H(\gamma_{y,t}\,|\,\gamma_t)\right) + \frac{t^2}{2}\sum_{i,j=1}^n \int_{\R^n} |\partial_{ij}f|^2\,d\gamma_{y^*,t}. \label{Gaussian-tilt-bound}
	\end{align}
\end{proposition}

Noting that $H(\gamma_{y,t}\,|\,\gamma_t)=|y|^2/2t$, a simple calculation shows that $y^*$ uniquely attains the supremum in \eqref{Gaussian-tilt-bound}.
The difference between \eqref{Gaussian-tilt-bound} and \eqref{MFapprox-refmeas} is that the former includes the diagonal terms $i=j$ in the sum. This is natural; an additively separable function $f(x)=\sum_{i=1}^n f_i(x_i)$ yields a product measure $P(dx)=Z^{-1}e^{f(x)}\gamma_t(dx)$, but it takes an \emph{affine} function $f$ for $P$ to be a Gaussian. Small off-diagonal derivatives $\partial_{ij}f$ can be naturally interpreted as meaning $f$ is close to being additively separable, but the full Hessian matrix $\nabla^2f$ must to be small in order for $f$ to be close to affine.

The above proposition is worth comparing with prior results based on gradient complexity. It was shown in \cite[Proposition 3.4, arXiv version]{augeri2} that if $f : \R^n \to \R$ is $C^1$ then
\begin{align}
	\log\int_{\R^n} e^f\,d\gamma_t &\le  \sup_{y \in \R^n}\left( \int_{\R^n} f\,d\gamma_{y,t} - H(\gamma_{y,t}\,|\,\gamma_t)\right) + \sqrt{2}\int_{\R^n} \sup_{y \in \R^n} \big(x \cdot \nabla f(y)\big) \, \gamma_t(dx). \label{Gaussian-tilt-bound-GW}
\end{align}
The last integral is ($\sqrt{t}$ times) the \emph{Gaussian mean-width} of the set  $\nabla f(\R^n)$. This estimate \eqref{Gaussian-tilt-bound-GW} has the advantage of applying to non-concave functions $f$, but it is only meaningful if $\nabla f$ is bounded. Proposition \ref{pr:tilt-Gaussian}, on the other hand, can accommodate non-Lipschitz but concave functions $f$.

\subsubsection{Generalization of the main theorem}\label{sec: generalization}

We briefly discuss how Theorem \ref{th:main-intro} can generalize beyond the strongly log-concave setting.
Essentially, strong log-concavity is needed only for the uniqueness claims and to justify the log-Sobolev and Poincar\'e inequalities as explained in Section \ref{se:proofideas}.
\emph{Uniqueness} of $Q^*$ is actually not essential if one is interested only in a bound like \eqref{MFapprox-mainthm}.
The \emph{existence} of an optimizer $Q^*$ is automatic, and it is not hard to show that it must satisfy the fixed point equation \eqref{fixedpoint-mainthm}, modulo technical conditions.
If it can be shown that $Q^*$ admits a strictly positive $C^2$ density, and that $P$ and $Q^*$ obey a log-Sobolev and Poincar\'e inequality, respectively, with constants $C_1$ and $C_2$, then the following bound can be proven as in Section \ref{se:proofideas}:
\begin{equation*}
0 \le R_f \le C_1 \E_{Q^*}\sum_{i=1}^n  \Var_{Q^*}(\partial_i f(X)\,|\,X_i)  \le 2C_1C_2\sum_{1\le i<j\le n} \E_{Q^*}[|\partial_{ij}f(X)|^2]. 
\end{equation*}
It is unclear if our assumed bound on $|f(x)|$ is needed or merely an artifact of our proof technique. We use the assumed bound on $|f(x)|$ in the proof of Theorem 1.1 only to show that $Q^*$ is strictly positive a.e., but this can be shown directly in many particular cases, such as when $f$ is symmetric.

\subsection{Outline of the paper}

In Section \ref{se:applications}, we will present in detail the three main applications of Theorem \ref{th:main-intro}, which pertain to Gibbs measures, high-dimensional Bayesian linear regression, and high-dimensional stochastic optimal control. The proof of Theorem \ref{th:main-intro}  is given in Section \ref{subsec:proofs main and log-concave}, followed by the proof of Corollary \ref{co:refmeas} in Section \ref{subsec: proofs-refprob}. Section \ref{subsec: proofs-tilts} contains the proof of Proposition \ref{pr:tilt-Gaussian}, while Section \ref{subsec: proofs-asymp-indep} contains the proofs of the subadditivity inequality  \eqref{W2subadditivity} and Corollary \ref{co:linearstat-concentration}. Finally, the proofs of the applications are given in Sections \ref{sec: proofs-gibbs} and \ref{se:control-proofs}.

\section{Applications} \label{se:applications}

\subsection{Gibbs measures with pairwise interactions} \label{se:Gibbs}

First, we study Gibbs measures  with pairwise interaction potentials  of the form \eqref{def:Gibbs-potential},
	where the following assumption holds:
	\begin{assumption} \label{ass:Gibbs}
		$V : \R \to \R$ is $C^2$ and $\kappa$-concave for some $ \kappa > 0$, $K : \R \to \R$ is even, $C^2$, and concave, and $J$ is a symmetric matrix with nonnegative entries and $J_{ii}=0$ for all $i=1,\ldots,n$. Assume there exists $ a,b,c \ge 0 $ and $0 \le d  < \kappa/2$  such that $ |V(x)| \le  c e^{d x^2} $ and $|K''(x)|^2 \le ae^{b|x|}$  hold for all $ x \in \R $.
	\end{assumption}
	Note since $K$ is even that there is no loss of generality in assuming that $J$ is zero on the diagonal.
	The most traditional \emph{mean field} setting is when $J_{ij}=1/n$ for all $(i,j)$, so that all particles interact equally, and there is a vast literature on the large-$n$ behavior; see \cite{chafai2014first,dupuis2020large} for some recent results and references.
	In general, the matrix $J$ represents \emph{disorder} or \emph{heterogeneous interactions}, and a common situation is when $J$ is the rescaled adjacency matrix of a graph. A notable strength of the non-asymptotic perspective of our work, and the theory of nonlinear large deviations more broadly, is that it can seamlessly handle this kind of heterogeneity.
	Gibbs measures with pairwise interactions on large graphs have been studied in many contexts, primarily on finite state space (see \cite{basak-mukherjee,biskup2003rigorous,dembo2010ising,dembo2014replica} and references therein). In the continuous context we study here, these Gibbs measures appear as invariant measures of locally interacting diffusion processes whose large-scale behavior has recently been the subject of active research  \cite{delattre2016note,oliveira2019interacting}.

To work toward applying Theorem \ref{th:main-intro} with $f$ as in \eqref{def:Gibbs-potential}, we first record the simple observation that $f$ is strongly concave under Assumption \ref{ass:Gibbs}. The proof of this and other results in Section \ref{se:Gibbs} are given in Section \ref{sec: proofs-gibbs}.

\begin{lemma} \label{lem: gibbs-log-concave}
	Define $f$ by \eqref{def:Gibbs-potential}, and suppose Assumption \ref{ass:Gibbs} holds. Then $f$ is $\kappa$-concave. 
\end{lemma}

	The following corollary will allow us to cover the case of unbounded $K''$, but only if we can control the barycenter of $Q^*$ in the sense that $\E_{Q^*}[X_i-X_j]=0$. This symmetry condition is justified in different ways in the following applications and is explained further in Section \ref{se:symmetry}. 
	
	\begin{corollary}\label{cor:quadratic}
		Define $f$ by \eqref{def:Gibbs-potential}, and suppose Assumption \ref{ass:Gibbs} holds.
		With $Q^*$ denoting the unique optimizer of \eqref{MFopt-mainthm}, assume further that $\E_{Q^*}[X_i-X_j]=0$. 
		Then 
		\begin{align*}
			R_f &\le \tr(J^2) a \kappa^{-2} e^{b^2/\kappa} .
		\end{align*}
	\end{corollary}

\begin{remark}\label{ref:compare}
	Corollary \ref{cor:quadratic} shows that $R_f=o(n)$ as long as $\mathrm{Tr}(J^2)=o(n)$.
	The assumption $\mathrm{Tr}(J^2)=o(n)$ has been used in the literature as a \emph{mean field condition} for quadratic interaction models, first in \cite[Theorem  1.1]{basak-mukherjee} and then in \cite[Theorem 4]{yan2020nonlinear}. Both cases are limited to measures with compact support. Moreover, in their setting, neither uniqueness of the optimizer nor convergence of the empirical measure hold in general. In contrast, in our setting we can allow measures of unbounded support, and we show both uniqueness of the optimizer and the convergence of the empirical measure in Theorems \ref{thm: regular graphs} and \ref{thm:graphons} below. On the other hand, our results require concavity assumptions which were not needed in \cite{basak-mukherjee,yan2020nonlinear}. 
\end{remark}

Using Corollary \ref{cor:quadratic}, one can study the weak law of large numbers of the empirical measure under $P$, by studying the corresponding weak law under the product measure $Q^*$. Under additional assumptions on the matrix $J$, the mean field optimization problem can be shown to converge as $n\to\infty$, allowing us to characterize the weak law under $P$ in terms of the limiting optimization problem. Below we illustrate this in two special cases.

\subsubsection{Doubly stochastic matrices}

In the following $n\to\infty$ results, note that the dependence of $f$, $P(dx)=Z^{-1}e^{f(x)}dx$ and $ J $ on $n$ is suppressed.

\begin{theorem} \label{thm: regular graphs}
	 Define $f$ by \eqref{def:Gibbs-potential}, and suppose Assumption \ref{ass:Gibbs} holds.
		Assume there exist $a,b \ge 0$ such that $|K''(x)|^2 \le ae^{b|x|}$ for all $x$.
	Assume further that the symmetric matrix $J$ is doubly stochastic (i.e., $\sum_{j=1}^nJ_{ij}=1$ for all $i$), and obeys the mean field condition $\tr(J^2)=o(n)$. Then we have the following conclusions: 
	\begin{enumerate}[(1)]
		\item \label{regular graphs conv}
		\begin{align}
			\lim_{n\to\infty}\frac{1}{n}\log \int_{\R^n} e^{f(x)}dx=\sup_{Q \in \P(\R)} \left( \int_{\R} V\,dQ + \frac12\int_{\R}\int_{\R} K(x-y)Q(dx)Q(dy) - H(Q)\right).\label{eq: regular graphs conv}
		\end{align}
		\item \label{regular graphs weak law}The supremum in \eqref{eq: regular graphs conv} is attained by a unique $Q \in \P(\R)$, and if $(X_1,\ldots,X_n) \sim P$ then
		\begin{align}
			\frac{1}{n}\sum_{i=1}^n\delta_{X_i} \to Q, \text{ weakly in law}. \label{eq: regular graphs weak law}
		\end{align} 
	\end{enumerate}
\end{theorem}

The above theorem applies when  $J=A/d$ and $A$ is the adjacency matrix of a $d$-regular graph. In this case we get $\tr(J^2)=n/d$, which is $o(n)$ as long as $d \to \infty$. The above theorem is similar in spirit to \cite[Theorem 2.1]{basak-mukherjee}, which dealt with Ising and Potts models, and a comment similar to Remark \ref{ref:compare} applies. Note that one cannot expect a mean field approximation to be valid in the sparsest (diluted) case, where $d$ stays bounded as $n\to\infty$. The framework of local weak convergence has proven to be successful in this context \cite{dembo2010gibbs}, and we refer also to \cite[Sections 2 and B]{lacker2019local} for continuous models encompassing the form studied here, and for a detailed derivation of the (folklore) limit of the empirical measure for locally convergent graph sequences, which requires uniqueness of the infinite-volume Gibbs measure on the limiting graph.

\subsubsection{Graphons}

Another case in which we can derive asymptotics of the log partition function is when the matrix $J$ converges to a graphon $W$ in cut metric. Below we introduce the relevant notions, deferring to \cite{borgs-Lp-1,borgs-Lp-2,borgs2008convergent,lovasz2012large} for additional background:

\begin{definition}
	Let $\mathcal{W}$ denote the space of all symmetric measurable functions from $[0,1]^2$ to $[0, \infty)$ which are integrable. For $W_1,W_2\in \mathcal{W}$, define the  strong cut (pseudo-)metric by 
	\[
	d_\square(W_1,W_2):=\sup_{S,T\subset [0,1]}\bigg|\int_{S\times T} \big(W_1(u,v)-W_2(u,v) \big)du dv\bigg|,
	\]
	and their weak cut (pseudo-)metric by
		\[
	\delta_\square(W_1,W_2):=\inf_{\varphi}d_\square(W_1,W_2^\varphi),
	\]
	where the infimum is over all invertible measure-preserving maps $\varphi : [0,1] \to [0,1]$, and $W_2^\varphi(u,v) := W_2(\varphi(u),\varphi(v))$.
	Given a symmetric matrix $A\in \R^{n\times n}$ with nonnegative entries, we define a function $W_{A} \in \mathcal{W}$ by setting  $W_A(u,v) \coloneqq A_{\lceil nu \rceil, \lceil nv \rceil}$.
	We say that a sequence of symmetric matrices $\{A_n\}$ \emph{converges in weak cut metric} to a function $W \in \mathcal{W}$ if $\delta_\square(W_{A_n},W)\to 0$.
\end{definition}

\begin{remark}
	Suppose $G_n$ is the adjacency matrix of an Erd\H{o}s-R\'enyi random graph on $n$ vertices with parameter $p_n$, such that $np_n\to \infty$. If $J_n=\frac{1}{np_n}G_n$, then $nJ_n$ converges in strong cut metric to the constant function $1$ (see \cite[Example 3.3.1]{borgs-Lp-2}). Similar convergences hold if $G_n$ arises from a stochastic block model, where the edge probability matrix has a block structure, in which case the limiting $W$ retains the same block structure. For more examples of convergent sequence of graphs in cut metric, we refer again to \cite{borgs-Lp-1,borgs-Lp-2,borgs2008convergent,lovasz2012large} and references therein.
\end{remark}

Let $\Punif$ denote the space of all probability measures on $[0,1]\times \R$ with uniform first marginal. 
Note that any $\mu \in \Punif$ admits the disintegration $\mu(du,dx)=du\mu_u(dx)$.

\begin{theorem}\label{thm:graphons}
	Define $f$ by \eqref{def:Gibbs-potential}, and suppose Assumption \ref{ass:Gibbs} holds.
		Assume there exist $a,b \ge 0$ such that $|K''(x)|^2 \le ae^{b|x|}$ for all $x$.
		Assume further that $V$ is even, $K$ is nonpositive, $\int_{\R} e^{V(x)}dx=1$, and $J=\{J_n\}$ is a sequence of matrices such that $\{nJ_n\}$ converges in weak cut metric to a function $W\in \mathcal{W}$. Assume also that $\tr(J_n^2) = o(n)$.

	\begin{enumerate}[(1)]
		\item \label{graphons conv} Defining the probability measure $\rho(dx)=e^{V(x)}dx$, we have
		\begin{align}
			\begin{split}
				&\lim_{n\to\infty}\frac{1}{n}\log \int_{\R^n} e^{f(x)}\,dx \\
				& \ \ = \sup_{\mu \in \Punif} \bigg( \frac12 \int_{([0,1] \times \R)^2}  \!\! W(u,v)K(x-y)\mu(du,dx)\mu(dv,dy) - \int_0^1 \!\!H(\mu_u\,|\,\rho)\,du\bigg). \label{eq: graphons conv}
			\end{split}
		\end{align}
		\item \label{graphons weak law} The supremum in \eqref{eq: graphons conv} is attained by a unique $\mu^* \in \Punif$, and if $(X_1,\ldots,X_n) \sim P$, then
		\begin{equation}
			\frac{1}{n}\sum_{i=1}^n\delta_{X_i} \to \int_0^1\mu^*_u \,du, \text{ weakly in law}. \label{eq: graphons weak law}
		\end{equation}
	\end{enumerate}
\end{theorem}

\begin{remark}
It follows from \cite[Propositions C.5 and C.15]{borgs-Lp-1} that the condition $\tr(J_n^2)=o(n)$ holds automatically if $J_n$ is the adjacency matrix of a simple graph $G_n=([n],E_n)$ multiplied by $n/(2|E_n|)$, and $nJ_n$ converges in cut metric. 
However, if $J_n$ is a general matrix, we need the added assumption $\tr(J_n^2)=o(n)$ in Theorem \ref{thm:graphons}.
\end{remark}

\subsubsection{On the symmetry of $Q^*$} \label{se:symmetry}

	This short section elaborates on conditions under which one can check that $\E_{Q^*}[X_i-X_j]=0$, which was needed in Corollary \ref{cor:quadratic}. The main two conditions we found are evenness and a weak form of permutation invariance. 
	
	\begin{definition}
		Let $S$ be a set of permutations of $[n]$. We say that $S$ is \emph{transitive} if for every $i,j \in [n]$ there exists $\pi \in S$ such that $\pi(i)=j$. We say also that a function $f$ on $\R^n$ is \emph{invariant under $S$} if $f(x_1,\ldots,x_n)=f(x_{\pi(1)},\ldots,x_{\pi(n)})$ for every $x \in \R^n$ and $\pi \in S$.
	\end{definition}
	
	\begin{lemma} \label{le:evenness}
		In the setting of Theorem \ref{th:main-intro}, the following implications hold:
		\begin{enumerate}[(1)]
			\item If $f$ is even, meaning $f(-x)=f(x)$ for all $x$, then $Q^*_i$ is even for each $i=1,\ldots,n$.
		
			\item Suppose $f$ is invariant under a transitive set of permutations. Then $Q^*_1=Q^*_2=\cdots=Q^*_n$. 
		\end{enumerate}
		In both cases, we  have $\E_{Q^*}[X_i-X_j]=0$ for all $i,j \in [n]$.
	\end{lemma}

	When $f$ is of the form \eqref{def:Gibbs-potential}, it is clear that $f$ is even if $K$ and $V$ are, and indeed $V$ is assumed even in Theorem \ref{thm: regular graphs} to enable an application of Lemma \ref{le:evenness}(1).
	We will not apply Lemma \ref{le:evenness}(2), but we find it interesting in its own right. For instance, (2) holds if $f$ is symmetric, i.e., invariant under all permutations. Another natural case covered by (2) is where $f$ is of the form \eqref{def:Gibbs-potential} and $J$ is a scalar multiple of the adjacency matrix of a \emph{vertex transitive} graph.

\subsection{High dimensional Bayesian linear regression}
Our next application is concerned with high dimensional Bayesian linear regression. Suppose we observe a set of data $ \left\{(y_i, X_i)\right\}_{i=1}^n $, where $ y_i \in \R $ and $ X_i \in \R^p $. Let $ y =  (y_1, \dots, y_n)^\top \in \R^n$ and $ \mathbf{X}^\top = (X_1, \dots, X_n) \in \R^{p \times n}$. Consider the linear regression model 
\begin{align*}
	y = \mathbf{X} \beta + \varepsilon, \quad \varepsilon \sim \gamma_{\sigma^2},
\end{align*}
where $ \gamma_{\sigma^2} $ denotes the Gaussian with mean $0$ and covariance matrix $\sigma^2I$. Here $\beta \in \R^p$ is the unknown parameter.

Following a Bayesian approach, assume that $\beta =(\beta_1,\ldots,\beta_p)^\top\stackrel{i.i.d.}{\sim}\pi$, where $\pi$ is a prior distribution on $\R$ with density proportional to $e^V\in L^1(\R)$ for some $V: \R \to \R$. The posterior density $\pi_{y,\mathbf{X}}$ of $\beta$ given $y$ and $ \mathbf{X} $ is then  proportional to $e^{f_{y,\mathbf{X}}}$, where
\begin{align*}
	f_{y,\mathbf{X}}(\beta) := \sum_{i=1}^pV(\beta_i) - \frac{1}{2\sigma^2}\left|y-\mathbf{X}\beta\right|^2 .
\end{align*}
The posterior distribution is the central object of inference in Bayesian statistics. Note that even though ${ \beta}$ has independent coordinates under the prior, the coordinates of ${\beta}$ are no longer independent under the posterior. Frequently, mean-field techniques are used to approximate such complex posterior distributions, including and beyond the set up of Bayesian linear regression (see \cite{alquier2016properties,blei2017variational,ray2021variational,wainwright2008graphical,zhang2020theoretical}  and references therein). In particular, it is useful to understand what conditions guarantee the validity of a mean field approximation, showing that the posterior is close to a product measure. Using Theorem \ref{th:main-intro}, the following corollary provides sufficient conditions under which the posterior is indeed mean-field. Leveraging this, it also derives a law of large numbers for the empirical measure under the true posterior distribution.

\begin{corollary} \label{co:bayes}
	Assume $V$ is $\kappa_1$-concave for some $\kappa_1 \in \R$, and that there exists $ c_1 \ge 0 $ and $0 \le c_2  < \kappa/2$ such that $ |V(x)| \le  c_1 e^{c_2 x^2} $ for all $x \in \R$.
	Set $J={\bf X}^\top{\bf X}\in \R^{p\times p}$, and assume that $J \ge \kappa_2 I$ for some $\kappa _2 \in \R$ such that $\kappa_1+\kappa_2\sigma^{-2} > 0$. Then 
	\begin{align}
			\sup_{y\in \R^n}\bigg|\log \int_{\R^p}e^{f_{y,{\bf X}}(\beta)}d\beta-	\sup_{Q \in \Pprod(\R^p)} \left( \int_{\R^p} f_{y,{\bf X}}\,dQ- H(Q)\right)\bigg| \le \frac{1}{(\kappa_1 \sigma^2 + \kappa_2 )^2}\sum_{1 \le i < j \le p}J_{ij}^2.  \label{eq: bayes-MF}
		\end{align}
		Moreover, for every $y \in \R^n$, the inner supremum in \eqref{eq: bayes-MF} is attained by a unique $Q^*_y \in \Pprod(\R^p)$, and for any 1-Lipschitz function  $\varphi : \R \to \R$, we have
		\begin{align}
				\sup_{y\in \R^n}\E_{\pi_{y,\mathbf{X}}}\Bigg[\bigg(\frac{1}{p}\sum_{i=1}^p\varphi(\beta_i)-\frac{1}{p}\sum_{i=1}^p \E_{Q^*_y}\big[\varphi(\beta_i)\big]\bigg)^2\Bigg] \le \frac{\sigma^2\Big(\kappa_1\sigma^2 + \kappa_2 + \sqrt{2\sum_{1 \le i < j \le p}J_{ij}^2}\Big)^2}{p(\kappa_1\sigma^2 + \kappa_2)^3}. \label{eq: bayes-LLN}
		\end{align}
\end{corollary}

The proof of this corollary is by a direct application of  Theorem \ref{th:main-intro} and Corollary \ref{co:linearstat-concentration}, and is hence omitted. Indeed, the concavity assumption on $V$ and the lower bound on $J$ ensure that $\nabla^2 f_{y,\mathbf{X}}(\beta) \le -( \kappa_1+\kappa_2\sigma^{-2})I$ for all $\beta$.

\begin{remark}
	The uniformity in $y$ in \eqref{eq: bayes-MF} implies that the mean field approximation continues to hold with high probability, under any distributional assumption on $y$. Note that when $ n,p \to \infty $ in any arbitrary manner, the right-hand side of  \eqref{eq: bayes-MF} and \eqref{eq: bayes-LLN} are $ o(p) $ and $ o(1) $ respectively, as long as $ \sum_{1 \le i < j \le p} J_{ij}^2 = o(p)$ when $ n,p \to \infty$. We also point out that the same conclusion as in \eqref{eq: bayes-MF} above was derived in \cite[Theorem 1]{mukherjee2021variational} using very different techniques, under the assumption that the prior distribution $\pi$ is compactly supported. In our setup, we allow the support to be non-compact, but instead assume that the prior distribution is strongly log-concave. One added advantage of our setup is that we also get the law of large numbers under no extra assumptions. 
\end{remark}

\subsection{Stochastic control} \label{se:control}

This section describes an application of Corollary \ref{co:refmeas} to a class of high-dimensional stochastic optimal control problems. Let $T > 0$, and let $g : \R^n \to \R$ be $C^2$ and concave. Consider the stochastic control problem
\begin{align}
	V_{\mathrm{orig}} \coloneqq \sup\,  \E\Bigg[ g(X_T) -  \frac{1}{2n} \sum_{i=1}^n \int_0^T |\alpha_i(t,X_t)|^2dt\Bigg], \label{control1}
\end{align}
where the supremum is over pairs $(\alpha,X)$, where  $\alpha=(\alpha_1,\ldots,\alpha_n) : [0,T] \times \R^n \to \R^n$ is a  measurable function and $X=(X^1,\ldots,X^n)$ a weak solution of the stochastic differential equation (SDE)
\begin{align}
	dX^i_t = \alpha_i(t,X_t)dt + dB^i_t, \qquad X^i_0=0, \ \ i=1,\ldots,n, \label{SDE}
\end{align}
defined on an arbitrary filtered probability space $(\Omega,\F,\FF,\PP)$,
satisfying also  $\int_0^T |\alpha(t,X_t)|^2  \,dt < \infty$ a.s.
Here $B=(B^1,\ldots,B^n)$ is an $n$-dimensional $\FF$-Brownian motion, and $X$ is required to be $\FF$-adapted. 
We call such a pair $(\alpha,X)$ \emph{admissible}.
There is a well known semi-explicit solution to \eqref{control1} which has come to be known as the \emph{F\"ollmer drift}, which we will discuss in Remark \ref{re:optimalcontrol} below.

We interpret $i=1,\ldots,n$ as the indices of different ``players," each facing an independent source of randomness $B^i$, and each choosing a control $\alpha_i$ which can depend on the full information of all $n$ players. Players ``cooperate" in the sense that $(\alpha_1,\ldots,\alpha_n)$ are chosen together to optimize \eqref{control1}.
When $g$ is of the form
\begin{align}
	g(x) = G\bigg(\frac{1}{n}\sum_{i=1}^n \delta_{x_i}\bigg), \ \ \text{for some } G : \P(\R)\to\R, \label{control-g=G(emp)}
\end{align}
we recover a well-studied class of problems which goes under the name \emph{mean field control} in the cooperative setting \cite{carmona2015forward}, or \emph{mean field games} in the competitive (Nash equilibrium) setting \cite{huang2006large,lasry2007mean}; see \cite{carmona2018probabilistic} for an overview.
In this setting, it is typically argued that $V_{\mathrm{orig}}$ converges to the value of a limiting ``mean field" control problem, and the optimal control $\widehat\alpha$ from this limiting problem can be used to construct \emph{distributed} controls $\alpha_i(t,x_1,\ldots,x_n) = \widehat\alpha(t,x_i)$ which are provably approximately optimal for the $n$-player problem for $n$ large.
This is a very desirable outcome, because distributed controls are much simpler (lower-dimensional).

Our results give a new non-asymptotic perspective on control problems of this form, by showing how to construct approximately optimal distributed controls for much more general $g$ than in \eqref{control-g=G(emp)}.
The link between  \eqref{control1} and the setting of Section \ref{sec: setup} is the formula 
\begin{align}
	V_{\mathrm{orig}} = \sup_{Q \in \P(\R^n)} \left(\int_{\R^n} g\,dQ - \frac{1}{n}H(Q\,|\,\gamma_T)\right) = \frac{1}{n}\log\int_{\R^n} e^{ng}\,d\gamma_T, \label{BBD}
\end{align}
where we recall that $\gamma_T$ denotes the centered Gaussian with covariance matrix $TI$.
This formula is essentially a well known consequence of Girsanov's theorem.\footnote{Experts might recognize a similarity with a famous formula often named after Bou\'e-Dupuis \cite{boue1998variational} or Borell \cite{borell2000diffusion}, though the form we present here is simpler because of our restriction to Markovian controls, whereas \cite{boue1998variational,borell2000diffusion} work with \emph{open-loop} controls, i.e., controls specified as arbitrary progressively measurable processes.}
The mean field approximation  also admits a natural control-theoretic interpretation. Define
\begin{align}
	V_{\mathrm{dstr}} := \sup \, \E\Bigg[ g(X_T) -  \frac{1}{2n} \sum_{i=1}^n \int_0^T |\alpha_i(t,X_t)|^2dt\Bigg], \label{control-distr}
\end{align}
where the supremum is now over admissible pairs $(\alpha,X)$ for which $\alpha=(\alpha_1,\ldots,\alpha_n)$ is of the form
\begin{align*}
	\alpha_i(t,x_1,\ldots,x_n) = \widehat\alpha_i(t,x_i),
\end{align*}
for some measurable $\widehat\alpha_i : [0,T] \times \R \to \R$, and also for which $X^1_t,\ldots,X^n_t$ are independent for each $t \in [0,T]$ (this second statement being redundant if the SDE \eqref{SDE} driven by this $\alpha$ is known to be unique in law).
Let us call any such pair $(\alpha,X)$ a \emph{distributed admissible pair}.
We will derive the following result from Corollary \ref{co:refmeas}, after first showing that $V_{\mathrm{dstr}}$ is nothing but the mean field approximation of \eqref{BBD}, in the sense that
\begin{align}
	V_{\mathrm{dstr}} = \sup_{Q \in \Pprod(\R^n)} \left(\int_{\R^n} g\,dQ - \frac{1}{n}H(Q\,|\,\gamma_T)\right) . \label{BBD-MF}
\end{align}

\begin{corollary} \label{co:stochcontrol}
	Let $g:\R^n\to\R$ be $C^2$ and concave, and let $T > 0$. Assume there exists $ c_1 \ge 0 $ and $0 \le c_2  < 1/2T $ such that $ |g(x)| \le  c_1 e^{c_2 |x|^2} $ for all $ x \in \R^n $. Define $V_{\mathrm{orig}}$ and $V_{\mathrm{dstr}}$ by \eqref{control1} and \eqref{control-distr}, respectively. Then the formulas \eqref{BBD} and \eqref{BBD-MF} hold, and
	\begin{align}
		0 \le V_{\mathrm{orig}} - V_{\mathrm{dstr}} \le  nT^2  \sum_{1 \le i < j \le n} \E_{Q^*}[|\partial_{ij}g(X)|^2], \label{ineq:controlbound}
	\end{align}
	where $Q^* = Q^*_1 \times \cdots \times Q^*_n \in \Pprod(\R^n)$ is the unique product measure with strictly positive density a.e. satisfying $g \in L^1(Q^*)$ and the fixed point equation
	\begin{align*}
		Q^*_i(dx_i) = Z_i^{-1} \exp \big(n\E_{Q^*}[g(X)\,|\,X_i=x_i]\big)\,\gamma_T(dx_i), \qquad  Z_i > 0, \ i=1,\ldots,n. 
	\end{align*}
\end{corollary}

	The proof is given in Section \ref{se:control-proofs}.
	Corollary \ref{co:stochcontrol} shows that distributed controls are approximately optimal for large $n$ if $n\|\sum_{i \neq j} \partial_{ij}g\|^2_\infty = o(1)$. As an example, if $g$ is of the form \eqref{control-g=G(emp)} and $G$ is twice continuously \emph{Wasserstein-} or  \emph{L-differentiable} in the sense of  \cite[Chapter 5.2]{carmona2018probabilistic}, then 
	\begin{align*}
		\partial_i g(x) = \frac{1}{n}D_mG\bigg(\frac{1}{n}\sum_{k=1}^n\delta_{x_k},x_i\bigg), \quad
		\partial_{ij} g(x) = \frac{1}{n^2}D_m^2G\bigg(\frac{1}{n}\sum_{k=1}^n\delta_{x_k},x_i,x_j\bigg), \ \ i \neq j.
	\end{align*}
	Hence, if $D_m^2G$ is bounded, then the right-hand side of \eqref{ineq:controlbound} is bounded by $T^2\|D_m^2G\|_\infty^2/2n$.
	
	\begin{remark} \label{re:optimalcontrol}
		In fact, the proof of Corollary \ref{co:stochcontrol} also yields an explicit characterization of the optimal distributed control in  \eqref{control-distr}, which we summarize  as follows.
		For a measure $Q \ll \gamma_T$, consider a process $X=(X_t)_{t \in [0,T]}$ such that $X_T \sim Q$ and the conditional law of the trajectory $(X_t)_{t \in [0,T]}$ given $X_T=x$ coincides with the law of the Brownian bridge from $0$ to $x$ on the time interval $[0,T]$. 
		This process might be called the \emph{Brownian} (or \emph{Schr\"odinger}) \emph{bridge with terminal law $Q$}.
		The associated control $\alpha$ is given by $\alpha(t,x)=\nabla_x\log \E[\frac{dQ}{d\gamma_T}(x+B_T-B_t)]$,
		as shown in full generality by F\"ollmer \cite{follmer1985entropy,follmer1986time}.
		Note that the associated SDE \eqref{SDE} may not be pathwise unique in general, but it always admits a weak solution $X$ with the law just described.
		The optimizer for the original control problem \eqref{control1} is nothing but the Brownian bridge with terminal law $P(dx)=Z^{-1}e^{ng(x)}\gamma_T(dx)$.
		Similarly, the optimizer for the distributed control problem \eqref{control-distr} is the Brownian bridge with terminal law  $Q^*$.
	\end{remark}

\begin{remark}

Proposition \ref{pr:tilt-Gaussian} admits a similar control-theoretic formulation in terms of \emph{deterministic} controls. Let $V_{\mathrm{det}}$ denote the value of the stochastic control problem \eqref{control1} but with the supremum limited to those admissible pairs $(\alpha,X)$ in which the control is non-random, i.e., $\alpha_i(t,x)=\tilde\alpha_i(t)$ for some $\tilde\alpha_i \in L^2[0,T]$. For these controls, $X_t$ is Gaussian with covariance matrix $tI$ for each $t \in [0,T]$. It can then be shown that
\begin{align*}
V_{\mathrm{det}} = \sup_{y \in \R^n } \left(\int_{\R^n} g\,d\gamma_{y,T} - \frac{1}{n}H(\gamma_{y,T}\,|\,\gamma_T)\right) = \sup_{y \in \R^n } \left(\int_{\R^n} g\,d\gamma_{y,T} - \frac{|y|^2}{2nT}\right),
\end{align*}
and Proposition \ref{pr:tilt-Gaussian} yields the following analogue of \eqref{ineq:controlbound}:
\begin{align*}
0 \le V_{\mathrm{orig}} - V_{\mathrm{det}} \le  \frac{nT^2}{2}  \sum_{i,j=1}^n \int_{\R^n}|\partial_{ij}g|^2 \,d\gamma_{y^*,T},
\end{align*}
where $y^* \in \R^n$ is the unique solution of $y^*=T\int_{\R^n}\nabla g\,d\gamma_{y^*,T}$.

\end{remark}

\section{Proof of the main theorem} \label{se:proof-mainthm}

The proofs will make use of the well known log-Sobolev and Poincar\'e inequalities for strongly log-concave measures, recalled here for convenience as we will use them in several parts of the paper.  The former is due to Bakry-\'Emery (see \cite{bakryemery} or \cite[Corollary 5.7.2]{bakry2013analysis}), and the latter is a consequence of the Brascamp-Lieb inequality \cite[Theorem 4.1]{brascamp2002extensions}.

\begin{theorem}[Log-Sobolev inequality] \label{th:logsobolev}
	If $h : \R^n \to \R$ is $C^2$ and $\kappa$-concave, and $R(dx)=e^{h(x)}dx$ is a probability measure,  then $R$ satisfies the log-Sobolev inequality,
	\begin{align*}
		H(Q\,|\,R) &\le \frac{1}{2\kappa}\int_{\R^n} \left|\nabla \log\frac{dQ}{dR}\right|^2\,dQ,
	\end{align*}
	for every $Q \in \P(\R^n)$ such that $Q \ll R$ and the weak gradient of $\log dQ/dR$ exists in $L^2(Q)$.
\end{theorem}

\begin{theorem}[Poincar\'e inequality] \label{th:poincare}
	If $h : \R^n \to \R$ is $\kappa$-concave, and $R(dx)=e^{h(x)}dx$ is a probability measure,  then $R$ satisfies the Poincar\'e inequality,
		\begin{align*}
		\Var_R(\varphi) \coloneqq \int_{\R^n} \varphi^2 \, dR- \left(\int_{\R^n} \varphi \, dR\right)^2 \le \frac{1}{\kappa} \int_{\R^n} |\nabla \varphi|^2 \, dR,
	\end{align*}
	for every continuously differentiable function $\varphi : \R^n \to \R$ in $L^1(R)$.
\end{theorem}

The above Poincar\'e inequality is normally stated with the additional assumptions that $h$ is $C^2$, which is easily removed by mollification by a Gaussian, and that $\varphi \in L^2(R)$, which can be weakened to $L^1(R)$ by monotone approximation, though both sides may be infinite.

We will also make use of the Gibbs variational principle, which is well known, but we give the proof as we need a non-standard form which is careful about edge cases. Recall our convention that $H(Q) :=\infty$ if $Q$ is not absolutely continuous or if $Q\log Q \notin L^1(\R^n)$.

\begin{theorem}[Gibbs variational principle] \label{th:gibbs}
Let $ f : \R^n \to \R \cup \{-\infty\} $ be measurable, bounded from above and such that $ Z \coloneqq \int_{\R^n} e^f \, dx \in (0, \infty)$. Define $P\in\P(\R^n)$ by $P(dx)=Z^{-1}e^{f(x)}\,dx$. Then 
	\begin{equation}
		\sup_{Q \in \P(\R^n) } \left( \int_{\R^n} f \,dQ - H(Q)\right) = \log Z \in (-\infty ,\infty), \label{gibbs-sup}
	\end{equation}
and the following are equivalent:
\begin{enumerate}
	\item $ H(P) < \infty $.
\item  The supremum in \eqref{gibbs-sup}
is attained uniquely by $P$.
\item There exists a maximizer in \eqref{gibbs-sup}.
\end{enumerate} 
\end{theorem}
\begin{proof} 
We first prove \eqref{gibbs-sup}. Since $ f $ is bounded from above, $ \int_{\R^n} f \, dQ \in [-\infty, \infty)$ is well-defined for all $ Q \in \Pprod(\R^n) $. 
We may thus restrict the supremum in \eqref{gibbs-sup} to those $Q$ with $H(Q) < \infty$. For $H(Q) < \infty$, we have the simple identity
\begin{equation}
\int_{\R^n} f \, dQ  - H(Q) = - H(Q \, |\, P) + \log Z.  \label{gibbs-identity}
\end{equation}
Therefore, 
\begin{equation*}
	\sup_{Q \in \P(\R^n) } \left( \int_{\R^n} f \,dQ - H(Q)\right)  = - \inf\Big\{H(Q\,|\,P) : Q \in \P(\R^n), \ H(Q) < \infty\Big\} + \log Z,  \label{gibbs-identity-2}
\end{equation*}
and it suffices to show that the infimum on the right-hand side is zero. We proceed by approximation. For each $k \in \N$, let $B_k \subset \R^n$ denote the centered ball of radius $k$, and define the probability density $Q_k = P1_{B_k}/P(B_k)$. Since $f$ is bounded from above, the density $Q_k$ is bounded and supported on the bounded set $B_k$. Thus $Q_k\log Q_k \in L^1(\R^n)$, or $H(Q_k) < \infty$, and we conclude that $H(Q\,|\,P) \le \liminf_k H(Q_k\,|\,P)$. Finally, since $P(B_k) \to 1$, 
\begin{align*}
	H(Q_k\,|\,P) = -\log P(B_k) \to 0.
\end{align*}
This proves the claim \eqref{gibbs-sup}.

Turning to the equivalence of (1--3), the implication (1) $ \Rightarrow $ (2) follows by taking $ Q = P $ in  \eqref{gibbs-identity}. The implication (2) $ \Rightarrow $ (3) is trivial. Lastly, for the implication (3) $ \Rightarrow $ (1), suppose $Q \in \P(\R^n)$ attains the supremum in \eqref{gibbs-sup}. We know from \eqref{gibbs-sup} that the supremum is not $ -\infty $, so $H(Q) < \infty$. Then, for any $R \in \P(\R^n)$ with $H(R) < \infty$, the identity \eqref{gibbs-identity} implies
\begin{align*}
	- H(R \, |\, P) + \log Z &= \int_{\R^n} f \, dR  - H(R) \le \int_{\R^n} f \, dQ  - H(Q) = - H(Q \, |\, P) + \log Z.
\end{align*}
Rearrange and minimize over $R$ to get
\begin{align*}
	H(Q\,|\,P) \le \inf\Big\{H(R\,|\,P) : R \in \P(\R^n), \ H(R) < \infty\Big\} = 0,
\end{align*}
where the last equality was shown just above while proving \eqref{gibbs-sup}. It follows that $H(Q\,|\,P)=0$, so $Q=P$, and $H(P)=H(Q)<\infty$. This completes the proof.

\end{proof}

\subsection{Proof of Theorem \ref{th:main-intro}} \label{subsec:proofs main and log-concave}
This section proves Theorem \ref{th:main-intro} in several parts, and we assume throughout that $f$ satisfies the assumptions therein.
Since $f$ is $C^2$ and $\kappa$-concave, 
\begin{align}
f(x) \le a - b|x|^2, \quad \text{ for all } x \in \R^n, \ \text{ where } a:=f(0)+\kappa^{-1} |\nabla f(0)|^2, \ b :=\kappa/4. \label{bound-f}
\end{align}
This implies that $Z:=\int_{\R^n}e^{f(x)}dx < \infty$, so $P(dx) = Z^{-1}e^{f(x)}dx$ is well defined. 
Moreover,  $f$ is bounded from above, so $\int_{\R^n} f\, dQ$ is well defined in $[-\infty,\infty)$ for every $Q \in \P(\R^n)$. 
Note lastly that $fe^f \in L^1(\R^n)$, or equivalently $H(P) < \infty$, which follows from the growth assumption on $|f|$ and the fact that the $\kappa$-log-concave measure $P$ satisfies $\int_{\R^n} e^{c|x|^2}P(dx) < \infty$ for each $c < \kappa/2$. (In fact, every absolutely continuous log-concave measure has finite entropy \cite[Theorem I.1]{bobkov2011entropy}.) We first establish some properties of the optimization and fixed point problems appearing in Theorem \ref{th:main-intro}.

\begin{lemma} \label{le:MFsup-finite}
	It holds that
	\begin{align}
		-\infty < \sup_{Q \in \Pprod(\R^n)} \left( \int_{\R^n} f\,dQ- H(Q)\right) < \infty,  \label{MFopt-pfsection}
	\end{align}
and any $Q^* \in \Pprod(\R^n)$ attaining the supremum satisfies $f \in L^1(Q^*)$. 
Also, equation \eqref{entproj-id1} is valid.
\end{lemma}
\begin{proof}
The Gibbs variational formula (Theorem \ref{th:gibbs}) implies that the supremum in \eqref{MFopt-pfsection} is no greater than $\log Z < \infty$. To see that it is not $-\infty$, note that $f$ is locally bounded because it is concave and real-valued. Hence, if $Q$ is any product measure with bounded support and finite entropy (such as the uniform measure on $[0,1]^n$), we can bound the supremum from below by $\int_{\R^n} f \, dQ- H(Q) > -\infty$. 
Now, if $Q^*$ is an optimizer, then $H(Q^*) < \infty$ and $\int_{\R^n} f dQ^* > -\infty$, the latter implying that $f \in L^1(Q^*)$ since $f$ is bounded from above.

To prove \eqref{entproj-id1}, note that the simple calculation \eqref{entproj-id0} is valid for any $Q \in \P(\R^n)$ with $H(Q) < \infty$, though both sides are $+\infty$ if and only if $\int_{\R^n} f\,dQ=-\infty$. Since $\int_{\R^n}f\,dQ$ always exists in $[-\infty,\infty)$, the supremum in \eqref{MFopt-pfsection} remains the same when restricted to those $Q$ with $H(Q) < \infty$. By infimizing \eqref{entproj-id0} over $Q \in \Pprod(\R^n)$ with finite entropy, we deduce that the left-hand side of \eqref{entproj-id1} is finite and equals $\inf\{H(Q\,|\,P) : Q \in \Pprod(\R^n), \, H(Q) < \infty\}$. To complete the proof, we claim that if $Q \in \Pprod(\R^n)$ satisfies $H(Q\,|\,P)<\infty$ and $H(Q)=\infty$, then there exists $Q_k \in \Pprod(\R^n)$ such that $H(Q_k) < \infty$ for each $k$ and $H(Q_k\,|\,P) \to H(Q\,|\,P)$. Indeed, define the probability density $Q_k=Q1_{B_k}/Q(B_k)$, where $B_k = [-k,k]^n$, for $k$ large enough that $Q(B_k) > 0$. Then 
\begin{align*}
H(Q_k\,|\,P) = \frac{1}{Q(B_k)}\int_{B_k} \log\frac{dQ}{dP}\,dQ - \log Q(B_k)
\end{align*}
is finite and converges to $H(Q\,|\,P)$ as $k\to\infty$. In particular, $\log (dQ_k/dP) \in L^1(Q_k)$. We also have  $\log P = f - \log Z \in L^1(Q_k)$ because $f$ is locally bounded and $Q_k$ has compact support. We deduce that $\log Q_k \in L^1(Q_k)$, or $H(Q_k) < \infty$, which completes the proof.

\end{proof}

The following proposition shows essentially that the fixed point problem \eqref{fixedpoint-mainthm} is the first order condition for optimality in \eqref{MFopt-mainthm}. This extends naturally to much more general settings, with $(\R^n,dx)$ replaced by a general $\sigma$-finite product measure space, but we will not need this.

\begin{proposition}[Optimality to fixed point] \label{pr:optimizer-to-fp}
	Suppose $Q^*=Q^*_1\times\cdots\times Q^*_n \in \Pprod(\R^n)$ attains the supremum in \eqref{MFopt-pfsection}.
	Then $ f \in L^1(Q^*) $ and $Q^*$ satisfies the fixed point equation
		\begin{align}
		\begin{split}
	Q_i(dx_i) &= Z_i^{-1} e^{\hat{f}_i(x_i)}\,dx_i, \qquad  \text{where } \hat{f}_i : \R \to \R \cup \{-\infty\} \text{ is defined by } \\
	\hat{f}_i(x_i) &:=  \int_{\R^{n-1}} f(x_1,\ldots,x_n) \, \prod_{j \neq i} Q^*_j(dx_j), \qquad i\in [n].
	\end{split} \label{fixedpoint-pfsection}
\end{align}
	\begin{proof}
		Note that $f \in L^1(Q^*)$ by Lemma \ref{le:MFsup-finite}.
		By assumption, $(Q^*_1,\ldots,Q^*_n)$ attains the supremum
		\begin{align*}
			\sup_{Q _1,\ldots,Q_n\in \P(\R) } \left( \int_{\R^n} f\,d(Q_1 \times \cdots \times Q_n) - H(Q_1 \times \cdots \times Q_n)\right).
		\end{align*}
		Clearly, $ \hat{f}_i(x_i)  = \E_{Q^*}[f(X)\,|\,X_i=x_i]$ for $ Q^*_i$-a.e.\ $ x_i \in \R $. Also,  it is well known that entropy tensorizes for product measures: $H(Q_1 \times \cdots \times Q_n) = \sum_{i=1}^nH(Q_i)$.
		From these and the tower property	it follows for each $i \in [n]$ that $Q^*_i$ attains the supremum
		\begin{align}
			S_i := \sup_{Q_i \in \P(\R) } \left( \int_{\R} \hat{f}_i\,dQ_i - H(Q_i)\right). \label{eq: S_i sup}
		\end{align}
We wish to invoke the Gibbs variational principle (Theorem \ref{th:gibbs}) to deduce that this supremum is uniquely attained by the probability measure with density proportional to $e^{\hat{f}_i}$, and thus $Q^*_i(dx_i) = Z_i^{-1}e^{\hat{f}_i(x_i)}\,dx_i$, which yields \eqref{fixedpoint-pfsection}.
It remains to carefully check the conditions of Theorem \ref{th:gibbs}.
We know that $Q^*_i$ attains the supremum \eqref{eq: S_i sup}, so we must just check that $Z_i \in (0,\infty)$.
Note that \eqref{bound-f} implies $f(x) \le a - b x_i^2$ for all $x \in \R^n$, and thus $\hat{f}_i(x_i) \le a - b x_i^2$ for all $x_i \in \R$,
which implies $ Z_i  = \int_\R e^{\hat{f}_i(x_i)} \, dx_i < \infty$. Next, recall from Lemma \ref{le:MFsup-finite} that $ f \in L^1(Q^*) $, so by Fubini's theorem, $Q_i^*(|\hat{f}_i| <\infty)=1$. Note that $Q^*_i$ is absolutely continuous since $H(Q_i^*) <\infty$. Hence, $\{|\hat{f}_i| < \infty\}$ has nonzero Lebesgue measure, and so $Z_i > 0$. 	

\end{proof}
\end{proposition}

\begin{lemma} \label{lem: maximizer unique}
There exists a unique  maximizer in \eqref{MFopt-pfsection}. 
\begin{proof}
We first prove existence. Recalling the identity \eqref{entproj-id1}, the optimizers of \eqref{MFopt-pfsection} are in one-to-one correspondence with the optimizers of $\inf_{Q \in \Pprod(\R^n)} H(Q\,|\,P)$. The latter exist because $\Pprod(\R^n)$ is a weakly closed subset of $\P(\R^n)$ and because $H(\cdot\,|\,P)$ has weakly compact sub-level sets.

		We next prove uniqueness. Let $Q^*=Q^*_1\times \cdots \times Q^*_n\in\Pprod(\R^n)$ denote any optimizer of \eqref{MFapprox-mainthm}.
		Define $G_1, G_2: (\P(\R))^n \to \R $ by 
		\begin{align*}
			G_1(Q_1, \dots, Q_n) \coloneqq \int_{\R^n} f(x_1, \dots, x_n) \prod_{i=1}^{n} Q_i(dx_i), \quad  G_2(Q_1, \dots, Q_n) \coloneqq H(Q_1 \times \cdots \times Q_n).
		\end{align*}
		That is, $Q^*$ is a maximizer of $G=G_1-G_2$, and we will show it must be the only one.
		Let $ Q  \in \Pprod(\R^n) $ be distinct from $Q^*$. We denote by $ M(t) = (M_1(t), \dots, M_n(t)) $ the displacement interpolations between the marginals, i.e.,
		\begin{align*}
			M_i(t) = Q^*_i \circ ((1-t) \mathrm{Id} + t T_i)^{-1}, 
		\end{align*}
		where $T_i : \R \to \R $ is the $ Q_i^* $-a.s.\ unique nondecreasing function satisfying  $ Q^*_i \circ T_i^{-1} = Q_i$. Since $Q^*_i$ and $Q_i$ are distinct, there exists $i$ such that $T_i$ is different from the identity map on a set with strictly positive $ Q^\ast_i$-measure.
		Writing out the expression of $G_1$,
		\begin{align*}
			G_1(M(t)) 
			&= \int_{\R^n} f\Big((1-t)x_1 + t T_1(x_1), \dots, (1-t)x_n + t T_n(x_n)\Big) \prod_{i=1}^{n} Q^*_i(dx_i),
		\end{align*}
		we see that $t \mapsto G_1(M(t))$ is strictly concave because $ f $ is strictly concave and $ Q \neq Q^*$. 
		Tensorization of entropy yields $ G_2(M(t))=\sum_{i=1}^nH(M_i(t))$, and it is well known that differential entropy is displacement convex \cite[Theorem 5.15(i)]{villani2003topics}. That is, $t\mapsto H(M_i(t))$ is convex for each $i$. We deduce that $t \mapsto G(M(t))$  is strictly concave. This proves uniqueness: if $Q$ were also an optimizer, then $G(M(1))=G(Q)=G(Q^*)=G(M(0))$ would imply $G(M(t)) > G(Q^*)$ for some $t \in (0,1)$.
\end{proof}
\end{lemma}	

\begin{remark}
We do not expect uniqueness in Lemma \ref{lem: maximizer unique} to hold under mere concavity of $f$.
The challenge is that the differential entropy functional is displacement convex, but not strictly so..
\end{remark}

In some of the following proofs, some shorthand notation will be useful. For $Q \in \P(\R^n)$, let us write $Q_{-i}$ for the marginal of $(X_j)_{j \neq i}$ under $Q$. For $x \in \R^n$ let us write $x_{-i}=(x_j)_{j \neq i}$ and, with some abuse of notation, $f(x)=f(x_i,x_{-i})$. 

\begin{lemma} \label{lem: optimizer log-concave}
	 If $Q \in \Pprod(\R^n)$ satisfies the fixed point equation \eqref{fixedpoint-pfsection}, then $Q$ is $\kappa$-log-concave.
\end{lemma}		
\begin{proof}
By \eqref{fixedpoint-pfsection}, the density of $Q$ is proportional to $e^{\hat{F}}$, where $\hat{F}(x)=\sum_{i=1}^n \hat{f}_i(x_i)$ and $ \hat{f}_i $ is given by \eqref{fixedpoint-pfsection}.  By the $ \kappa $-concavity of $ f $,  for any $ y, z \in \R $ and $t\in [0,1] $,  we have
	\begin{align*}
		\hat{f}_i &(tz + (1- t) y ) + \frac{\kappa}{2}(tz + (1-t)y)^2 \\
		&= \int_{\R^{n-1}} \Big[f(tz + (1 - t) y, x_{-i})+ \frac{\kappa}{2}(tz + (1-t)y)^2\Big] Q_{-i} (x_{-i})dx_{-i} \\
		&\ge \int_{\R^{n-1}} \Big[tf(z, x_{-i})+ t\frac{\kappa}{2}z^2 + (1-t)f(y, x_{-i}) + (1-t)\frac{\kappa}{2}y^2\Big] Q_{-i} (x_{-i})dx_{-i}  \\
		&= t\hat{f}_i(z) + t\frac{\kappa}{2}z^2 + (1-t)\hat{f}_i(y)  + (1-t)\frac{\kappa}{2}y^2.
	\end{align*}
	This shows that $\hat{f}_i$ is $\kappa$-concave, and thus so is $\hat{F}$.
\end{proof}

The next proposition, in conjunction with Proposition \ref{pr:optimizer-to-fp}, shows that the optimizers of \eqref{MFopt-pfsection} and the solutions of the fixed point problem \eqref{fixedpoint-mainthm} are exactly the same.

\begin{proposition}[Fixed point to optimality] \label{pr:fp-to-optimizer}
	 Let $ Q \in \Pprod(\R^n)$ satisfy $ f \in L^1(Q) $ and the fixed point problem \eqref{fixedpoint-pfsection}.
Then $Q$ has strictly positive density a.e.\ and is a maximizer of \eqref{MFopt-pfsection}.
\end{proposition}		
\begin{proof}
We first show that  $ Q $ has strictly positive density a.e. Since $ Q = Q_1 \times \cdots \times Q_n $ satisfies the fixed point equation \eqref{fixedpoint-pfsection}, each $ Q_i $ has a  density with exponent
\begin{align*} 
	\hat{f}_i(x_i) = \int_{\R^{n-1}} f(x_1,\ldots,x_n) \, \prod_{j \neq i} Q_j(dx_j)  \ge -c_1 e^{c_2x_i^2} \prod_{j \neq i}  \int_{\R} e^{c_2x_j^2}  Q_j(x_j) \, dx_j
\end{align*}
for every $ x_i \in \R $. From Lemma \ref{lem: optimizer log-concave} we know that $Q$ is $ \kappa $-log-concave.
Since $c_2 < \kappa/2$, we deduce that $\int_{\R} e^{c_2x_j^2}  Q_j(x_j) < \infty$.
Thus $ \hat{f}_i (x_i) > -\infty$ for all $ x_i \in \R $. 

Define $G(R) \coloneqq \int_{\R^n} f \, dR- H(R)$ for $R \in \P(\R^n)$. Let $Q^*$ be an optimizer of $\sup\{G(R) : R \in \Pprod(\R^n)\}$, which exists uniquely by Lemma \ref{lem: maximizer unique}. By Proposition \ref{pr:optimizer-to-fp}, we have $f \in L^1(Q^*)$, and  $Q^*$ satisfies the fixed point equation \eqref{fixedpoint-pfsection}. The argument given in the previous paragraph  implies that $Q^*$ has a strictly positive density a.e.
To complete the proof, we must show that $G(Q) \ge G(Q^*)$.
	
For $i=1,\ldots,n$, let $T_i : \R \to \R$ denote the unique nondecreasing function satisfying $Q_i \circ T_i^{-1}=Q^*_i$, and define $M_i(t) = Q_i \circ ((1-t)\mathrm{Id} + tT_i)^{-1}$. Let $M(t) = M_1(t) \times \cdots \times M_n(t)$, so that $G(M(t))=g_1(t)-g_2(t)$, where
	\begin{align*}
		g_1(t) &:= \int_{\R^n} f\big((1-t)x_1 + tT_1(x_1),\ldots,(1-t)x_n + tT_n(x_1)\big) \prod_{i=1}^n Q_i(dx_i), \\
		g_2(t) &:= H\big(M_1(t) \times \cdots \times M_n(t)\big) = \sum_{i=1}^nH(M_i(t)).
	\end{align*}
	Let us write $g^{\prime +}$ for the right-derivative of a real-valued function $g$, when it exists. 
	Note that $T_i$ is a.e.\ differentiable, as it is monotone.
	Using \cite[Theorem 5.30]{villani2003topics}, we may compute the right-derivatives at zero as
	\begin{align*} 
		g_1^{\prime +}(0) &= \sum_{i=1}^n \int_{\R^n}   \partial_i f (x) \big( T_i (x_i) -x_i \big) Q(x) dx, 
		 \\
		g_2^{\prime +}(0) &= - \sum_{i=1}^n \int_{\R}  \left(T_i'(x_i) - 1\right)Q_i(x_i) dx_i. 
	\end{align*}
	We wish to rewrite both terms in more useful forms. 
	
	We first claim that
	\begin{align}
		\int_{\R^n}   \partial_i f (x) \big( T_i (x_i) -x_i \big) Q(x) dx &= \int_{\R}  \hat{f}'_i (x_i) \big( T_i (x_i) - x_i \big) Q_i(x_i) dx_i, \label{eq:g1prime-claim1}
	\end{align}
	where $ \hat{f}_i $ is defined as in \eqref{fixedpoint-pfsection}.
	To see this, note that  $ \hat{f}_i(x_i) = \E_{Q}[f(x_i,X_{-i})] $ for all $ x_i \in \R $, so 
	\begin{align*}
	\hat{f}^{\prime +}_i (x_i) &= \lim_{h\downarrow 0} h^{-1}\E_Q[f(x_i+h,X_{-i}) - f(x_i,X_{-i})].
\end{align*}
By the concavity of $f$, the difference quotient $[f(x_i+h,X_{-i}) - f(x_i,X_{-i})]/h$ increases as $h \downarrow 0$,	and it is bounded from below for $0 < h \le h_0$ by $[f(x_i+h_0,X_{-i}) - f(x_i,X_{-i})]/h_0$, which has finite $Q$-expectation for a.e.\ choice of $h_0 > 0$ by Fubini's theorem since $f \in L^1(Q)$. 
	Hence, by monotone convergence,
	\begin{equation}
		\hat{f}^{\prime +}_i (x_i) = \E_Q[ \partial_i f (x_i, X_{-i})].\label{eq:g1prime-claim2}
	\end{equation}
	Moreover, this quantity is finite and nonincreasing in $x_i$ because $\hat{f}_i$ is a concave real-valued function. In addition, $\hat{f}'_i=\hat{f}^{\prime +}_i$ a.e. since concave functions are a.e.\ differentiable.
	Using \eqref{eq:g1prime-claim2}, we see that the right-hand side of \eqref{eq:g1prime-claim1} equals $\E_Q[\E_Q[\partial_i f (X) \,|\, X_i](T_i(X_i)-X_i)]$, which yields \eqref{eq:g1prime-claim1}.

	We next integrate by parts to get
	\begin{align}
		-\int_{\R}  \left(T_i'(x_i) - 1\right)Q_i(x_i) dx_i =  \int_\R \left(T_i(x_i) - x_i\right)Q_i'(x_i) \, dx_i. \label{eq:g2prime-claim1}
	\end{align}
	To justify this carefully, we use Lebesgue-Stieltjes integration by parts: Note that the probability density function of $Q_i$ is absolutely continuous because it is proportional to $e^{\hat{f}_i}$, and $\hat{f}_i$ is absolutely continuous as a concave function. Let $ F_{Q_i} $ and $ F_{Q^*_i} $ denote the CDFs of $ Q_i $ and $ Q^*_i $ respectively. Recalling that $T_i=F_{Q^*_i}^{-1} \circ F_{Q_i}$ is the monotone map pushing $Q_i$ forward to $Q^*_i$, and that both $Q_i$ and $Q^*_i$ admit strictly positive densities, the function $T_i$ is absolutely continuous. Hence, there is no jump term in the integration by parts, and we must only show that the boundary terms vanish. For this it suffices to show that there exist sequences $x^{\pm}_n \to \pm\infty$ such that
	\begin{align*}
		\lim_{n \to \infty} (T_i(x^\pm_n)-x^\pm_n)Q_i(x^\pm_n) = 0.
	\end{align*}
	If this were not the case, it would imply that $|T_i(x)-x|Q_i(x) \le (|T_i(x)|+|x|)Q_i(x)$ is bounded away from zero for $|x|$ sufficiently large. This would in turn imply that $\int_\R (|T_i(x_i)|+|x_i|)Q_i(x_i) dx_i= \infty$, contradicting
the fact that
\begin{align*}
\int_\R (|T_i(x_i)|+|x_i|)Q_i(x_i)dx_i = \int_\R |x_i|Q^*_i(x_i)dx_i + \int_\R |x_i|Q_i(x_i)dx_i < \infty.
\end{align*}
Both integrals are finite because $ Q_i $ and $Q^*_i$ are $ \kappa $-log-concave by Lemma \ref{lem: optimizer log-concave} and thus admit finite moments of every order. 
	With \eqref{eq:g2prime-claim1} and \eqref{eq:g1prime-claim1} now justified, we see that the right-derivative of $G(M(t))$ at $t=0$ is
	\begin{align*}
		g_1^{\prime +}(0) - g_2^{\prime +}(0) &= \sum_{i=1}^n\int_\R  \big(\hat{f}'_i (x_i)Q_i(x_i) - Q_i'(x_i)\big) \big( T_i (x_i) - x_i \big) dx_i.
	\end{align*}
	This is in fact zero, because $Q_i$ is proportional to $e^{\hat{f}_i}$. We saw in the proof of Lemma \ref{lem: maximizer unique} that $G(M(t))$ is concave. Since we now know that it has vanishing right-derivative at $t=0$, it follows that $G(M(1)) \le G(M(0))$. That is, $G(Q^*) \le G(Q)$, which completes the proof.
\end{proof}

\begin{proof}[\textbf{Proof of Theorem \ref{th:main-intro}}]
	Let $S_{\mathrm{opt}}$ denote the set of maximizers  in \eqref{MFopt-pfsection}, and let $S_{\mathrm{fix}}$ denote the set of $Q^* \in \Pprod(\R^n)$ satisfying $f \in L^1(Q^*)$ and the fixed point equation \eqref{fixedpoint-pfsection}. 
	Proposition \ref{pr:optimizer-to-fp} shows that $S_{\mathrm{opt}} \subset S_{\mathrm{fix}}$. Proposition \ref{pr:fp-to-optimizer} shows conversely that $S_{\mathrm{opt}} \supset S_{\mathrm{fix}}$, so in fact $S_{\mathrm{opt}} = S_{\mathrm{fix}}$. Lemma \ref{lem: maximizer unique} shows that this set is a singleton. Its unique element $Q^*$ is $\kappa$-log-concave by Lemma \ref{lem: optimizer log-concave} and has strictly positive density a.e.\ by Proposition  \ref{pr:fp-to-optimizer}.
	This proves claims (\ref{main.fp}--\ref{main.unique}) of Theorem \ref{th:main-intro}.

	To prove \eqref{main.bound}, recall the identity \eqref{entproj-id1}, which shows that 
	\begin{align*}
		R_f = \log \int_{\R^n} e^{f(x)}\,dx - \sup_{Q \in \Pprod(\R^n)}\left(\int_{\R^n} f\,dQ - H(Q)\right) = H(Q^*\,|\,P).
	\end{align*}
	The $\kappa$-log-concavity of $P$ and the log-Sobolev inequality (Theorem \ref{th:logsobolev}) imply
	\begin{align*}
		H(Q^*\,|\,P) &\le \frac{1}{2\kappa}\int_{\R^n}\left|\nabla \log \frac{dQ^*}{dP}\right|^2\,dQ^*. 
	\end{align*}
	Since $Q^*=Q^*_1\times \cdots \times Q^*_n$ is a product measure, we have $\partial_i \log Q^*(x) = \partial_i\log Q^*_i(x_i)$ for $x \in \R^n$  and note that the derivative exists almost everywhere because $\log Q^*_i$ is concave. We saw in \eqref{eq:g1prime-claim2} in the proof of Proposition \ref{pr:fp-to-optimizer} that the following identity is valid  for almost every $x_i \in \R$, with the expectation on the right-hand side being finite:
	\begin{align*}
		\partial_i\log Q^*_i(x_i) = \partial_i\E_{Q^*}[ f(X)\,|\,X_i=x_i] = \E_{Q^*}[\partial_i f(X)\,|\,X_i=x_i].
	\end{align*}
	Thus,
	\begin{align*}
		H(Q^*\,|\,P) &\le \frac{1}{2\kappa} \int_{\R^n} \sum_{i=1}^n \Big| \partial_i \log Q^*_i(x_i) - \partial_i f(x)\Big|^2\,Q^*(dx) \\
		&= \frac{1}{2\kappa} \E_{Q^*}\sum_{i=1}^n \left(\E_{Q^*}[\partial_i f(X)\,|\,X_i] - \partial_if(X)\right)^2  \\
		&= \frac{1}{2\kappa}\E_{Q^*}\sum_{i=1}^n \Var_{Q^*}(\partial_i f(X)\,|\,X_i).
	\end{align*}
This yields the first bound in \eqref{MFapprox-mainthm}.
	Recall that $Q^*_{-i}$ denotes the law of $(X_j)_{j \neq i}$, which equals the conditional law of $(X_j)_{j \neq i}$ given $X_i$ under $Q^*$ by independence. The measure $Q^*_{-i}$ is $\kappa$-log-concave because $Q^*_j$ is for each $j$. Hence, it obeys a Poincar\'e inequality (Theorem \ref{th:poincare}), $\Var_{Q^*_{-i}}(\varphi) \le \kappa^{-1} \int_{\R^{n-1}} |\nabla \varphi|^2 \, dQ^*_{-i},$ for any $C^1$ function $\varphi \in L^1(Q^*_{-i})$. Applying this to $\partial_if$ with coordinate $i$ fixed, 
	\begin{equation*}
		\Var_{Q^*}(\partial_i f(X)\,|\,X_i) \le \frac{1}{\kappa}\sum_{j \neq i}\E_{Q^*}[|\partial_{ij}f(X)|^2\,|\,X_i]. 
	\end{equation*}
Complete the proof of the second inequality of \eqref{MFapprox-mainthm} by using the tower property to get
	\begin{equation*}
		\frac{1}{2\kappa}\E_{Q^*}\sum_{i=1}^n\Var_{Q^*}(\partial_i f(X)\,|\,X_i) \le \frac{1}{2\kappa^2}  \sum_{i=1}^n\sum_{j \neq i}\E_{Q^*}[|\partial_{ij}f(X)|^2] . \qedhere
	\end{equation*}
\end{proof}

\subsection{Proof of Corollary \ref{co:refmeas}} \label{subsec: proofs-refprob}

	Let $f(x) := g(x) + \sum_{i=1}^nV_i(x_i) $.
	Then $ \int_{\R^n} e^g\,d\rho = \int_{\R^n} e^{f(x)}\,dx$, and the concavity of $g$ and $\kappa$-concavity of $V_i$ imply that $f$ is $\kappa$-concave. 
	Note also that for any $Q \in \Pprod(\R^n)$,
	\begin{equation*}
		\int_{\R^n} g\,dQ- H(Q\,|\,\rho) = \int_{\R^n} f\,dQ - H(Q).
	\end{equation*}
This shows that the optimization problems \eqref{MFopt-mainthm} and \eqref{MFopt-refmeas} are the same.
Moreover, the fixed point problems \eqref{fixedpoint-refmeas} and \eqref{fixedpoint-mainthm} admit exactly the same solutions: $Q^*_i$ solves \eqref{fixedpoint-mainthm} if and only if it solves \eqref{fixedpoint-refmeas}. With these identifications, applying Theorem \ref{th:main-intro} to $f$ immediately proves
claims (1--3) of Corollary \ref{co:refmeas}.
Finally, 
with $Q^*_i$  solving \eqref{fixedpoint-pfsection} (or equivalently \eqref{fixedpoint-refmeas}), we have
\begin{align*}
R^\rho_g = R_f &\le \frac{1}{\kappa^2}\sum_{1 \le i < j \le n} \E_{Q^*}[|\partial_{ij}f(X)|^2] = \frac{1}{\kappa^2}\sum_{1 \le i < j \le n} \E_{Q^*}[|\partial_{ij}g(X)|^2 ],
\end{align*}
	because $\partial_{ij}f=\partial_{ij}g$ for all $i \neq j$. This proves  claim (4) of Corollary \ref{co:refmeas}. \hfill\qedsymbol

\subsection{Proof of Proposition \ref{pr:tilt-Gaussian}} \label{subsec: proofs-tilts}

Note that $\int_{\R^n} f(x+y)\,\gamma_t(dx) < \infty$ for each $y \in \R^n$ by the growth assumption on $f$.
	The function 
	\begin{equation*}
		y \mapsto \int_{\R^n} f\,d\gamma_{y,t} - H(\gamma_{y,t}\,|\,\gamma_t)
		= \int_{\R^n} f(x+y)\,\gamma_t(dx) - \frac{1}{2t}|y|^2
	\end{equation*}
is $(1/t)$-concave and thus bounded from above. It admits a unique maximizer obtained by setting the gradient equal to zero; the first order condition is precisely \eqref{Gaussian-tilt-eq}.
	Let $P(dx)=Z^{-1}e^{f(x)}\gamma_t(dx)$.
	The simple identity
	\begin{align*}
		\log \int_{\R^n} e^{f}\,d\gamma_t - \left(\int_{\R^n} f\,d\gamma_{y,t} - H(\gamma_{y,t}\,|\,\gamma_t)\right) = H(\gamma_{y,t}\,|\,P),
	\end{align*}
	valid for all $y \in \R^n$, implies that
	\begin{align*}
		\log \int_{\R^n} e^{f}\,d\gamma_t - \sup_{y \in  \R^n }\left(\int_{\R^n} f\,d\gamma_{y,t} - H(\gamma_{y,t}\,|\,\gamma_t)\right) = \inf_{y \in \R^n}H(\gamma_{y,t}\,|\,P).
	\end{align*}
	The right-hand side is equal to $H(\gamma_{y^*,t}\,|\,P)$.
	The measure $P$ is $(1/t)$-log-concave, so we may use the log-Sobolev inequality (Theorem \ref{th:logsobolev}) to get
	\begin{align*}
		H(\gamma_{y^*,t}\,|\,P) &\le \frac{t}{2} \int_{\R^n} \left|\nabla \log\frac{d\gamma_{y^*,t}}{dP}\right|^2\,d\gamma_{y^*,t} = \frac{t}{2} \int_{\R^n} \left|\nabla \log\frac{d\gamma_{y^*,t}}{d\gamma_t}-\nabla \log\frac{dP}{d\gamma_t}\right|^2\,d\gamma_{y^*,t} \\
		&= \frac{t}{2}\int_{\R^n} \left| \frac{1}{t} y^* - \nabla f(x)\right|^2\,\gamma_{y^*,t}(dx) \\
		&= \frac{t}{2}\sum_{i=1}^n\Var_{\gamma_{y^*,t}}(\partial_i f),
	\end{align*}
	where the last step follows from \eqref{Gaussian-tilt-eq}.
Using the Gaussian Poincar\'e inequality (or Theorem \ref{th:poincare}), this is bounded by 
the second term on the right-hand side of \eqref{Gaussian-tilt-bound}. \hfill\qedsymbol

\subsection{Asymptotic independence} \label{subsec: proofs-asymp-indep}

\begin{proof}[Proof of first inequality in \eqref{W2subadditivity}]
	Let $P,Q \in \P(\R^n)$. Let $k_1,\ldots,k_m$ be positive integers summing to $n$. Suppose $P_1,\ldots,P_m$ are the marginals of $P$ on $\R^{k_1},\ldots,\R^{k_m}$, and define the marginals $Q_1,\ldots,Q_m$ similarly. Then
	\begin{align*}
		\sum_{i=1}^m \mathcal{W}_2^2(P_i,Q_i) \le \mathcal{W}_2^2(P,Q).
	\end{align*}
	Indeed, to prove this, let $(X,Y)$ be an optimal coupling of $(P,Q)$. Let $X_i$ be the $\R^{k_i}$ coordinate, for $i=1,\ldots,m$, and similarly define $Y_i$. Then $(X_i,Y_i)$ is a coupling of $(P_i,Q_i)$, and so
	\begin{align*}
		\mathcal{W}_2^2(P,Q) &= \E\left[|X-Y|^2\right] = \E\left[\sum_{i=1}^m  |X_i-Y_i|^2 \right]\ge \sum_{i=1}^m \mathcal{W}_2^2(P_i,Q_i).
	\end{align*}
	
	Now, let $1 \le k \le n$, and let $m=\lfloor n/k\rfloor$. Let $\Pi$ be the set of vectors $(S_1,\ldots,S_m)$ of disjoint $k$-element subsets of $[n]$. Let $Q_{S_i}$ and $P_{S_i}$ denote the corresponding marginals, on those coordinates in $S_i \subset [n]$. Note that $\mathcal{W}_2^2(P_{S_i},Q_{S_i})$ does not depend on the order of the elements of $S_i$. Then
	\begin{align*}
		\sum_{i=1}^m \mathcal{W}_2^2(P_{S_i},Q_{S_i}) \le \mathcal{W}_2^2(P_{S_1 \cup \cdots \cup S_m},Q_{S_1 \cup \cdots \cup S_m}) \le \mathcal{W}_2^2(P,Q).
	\end{align*}
	If $(S_1,\ldots,S_m)$ is chosen uniformly at random from $\Pi$ and $ i $ is chosen uniformly at random from $[m]$, then the marginal law of $S_i$ is the same as the law of a uniformly random choice of $k$-element subset of $[n]$.
	In particular,
	\begin{align*}
		\frac{1}{\binom{n}{k}}\sum_{S \subset [n], \, |S|=k} \mathcal{W}_2^2(P_{S},Q_{S}) = \frac{1}{|\Pi|}\sum_{(S_1,\ldots,S_m) \in \Pi}\frac{1}{m}\sum_{i=1}^m \mathcal{W}_2^2(P_{S_i},Q_{S_i}).
	\end{align*}
	Combining the two previous inequalities yields
	\begin{equation*}
		\frac{1}{\binom{n}{k}}\sum_{S \subset [n], \, |S|=k} \mathcal{W}_2^2(P_{S},Q_{S}) \le \frac{1}{m} \mathcal{W}_2^2(P,Q) = \frac{1}{\lfloor n/k\rfloor} \mathcal{W}_2^2(P,Q). \qedhere
	\end{equation*}
\end{proof}

\begin{proof}[Proof of Corollary \ref{co:linearstat-concentration}]
	By the triangle inequality, the square root of the left-hand side of \eqref{ineq:weakLLN} is no more than $A_1+A_2$,
	where we define
	\begin{align*}
		A_1 &:= \E_P\left[ \left(\frac{1}{n}\sum_{i=1}^n\varphi(X_i) - \frac{1}{n}\sum_{i=1}^n\E_P[\varphi(X_i)]\right)^2\right]^{1/2}, \\
		A_2 &:=  \left|\frac{1}{n}\sum_{i=1}^n (\E_P[\varphi(X_i)] - \E_{Q^*}[\varphi(X_i)])\right|.
	\end{align*}
	Recall that $|\varphi'| \le 1$.
	Using Kantorovich duality and  \eqref{W2subadditivity} with $k=1$,
	\begin{align*}
		A_2^2 \le \frac{1}{n}\sum_{i=1}^n\mathcal{W}_1^2(P_i,Q^*_i) \le  \frac{1}{n}\sum_{i=1}^n\mathcal{W}_2^2(P_i,Q^*_i)   \le  \frac{2R_f}{\kappa n}.
	\end{align*}
	Apply the Poincar\'e inequality (Theorem \ref{th:poincare}) to the function $x \mapsto (1/n)\sum_{i=1}^n \varphi(x_i)$ to get
	\begin{align*}
		A_1^2 &= \Var_P\left( \frac{1}{n}\sum_{i=1}^n\varphi(X_i)\right) \le \frac{1}{\kappa n^2}\sum_{i=1}^n \E_P[| \varphi'(X_i)|^2] \le \frac{1}{\kappa n}.
	\end{align*}
	Combine these two bounds to complete the proof.
\end{proof}

\section{Gibbs measure proofs} \label{sec: proofs-gibbs}

This section proves the results of Section \ref{se:Gibbs}. Throughout, the function $f : \R^n \to \R$ is defined as in \eqref{def:Gibbs-potential} 
 and satisfies Assumption \ref{ass:Gibbs}. 

\noindent\textbf{Proof of Lemma \ref{lem: gibbs-log-concave}.}
Compute two derivatives to find, for all $i \neq j$,
\begin{align*}
	\partial_{ii}f(x) &= V''(x_i) + \sum_{j \neq i} J_{ij}K''(x_i-x_j), \qquad  
	\partial_{ij}f(x) = -J_{ij}K''(x_i-x_j).
\end{align*}
Hence, for any $x,z \in \R^n$,
\begin{align*}
	z^\top\nabla^2f(x) z &= \sum_{i,j=1}^n z_iz_j\partial_{ij}f(x) = \sum_{i=1}^n z_i^2V''(x_i) + \sum_{i,j=1}^n \big(z_i^2 - z_iz_j\big) J_{ij} K''(x_i-x_j).
\end{align*}
Using the evenness of $K''$ and the symmetry of $J$, 
\begin{align*}
	\sum_{i,j=1}^n \big(z_i^2 - z_iz_j\big) J_{ij} K''(x_i-x_j) &= \frac12\sum_{i,j=1}^n (z_i - z_j)^2 J_{ij} K''(x_i-x_j).
\end{align*}
Since $K'' \le 0$ and $J_{ij} \ge 0$, we find that this quantity is nonpositive. By $\kappa$-concavity of $V$,
\begin{align*}
	z^\top\nabla^2f(x) z &\le \sum_{i=1}^n z_i^2V''(x_i) \le -\kappa |z|^2,
\end{align*}
which shows that $f$ is $\kappa$-concave. \hfill \qed

\noindent\textbf{Proof of Corollary \ref{cor:quadratic}.}
Note that $f$ is $C^2$ and $\kappa$-concave.
	Also, the assumptions on $ |V|$ and $| K'' |$ in Assumption \ref{ass:Gibbs} clearly imply that $ |f| $ satisfies the growth assumption in Theorem \ref{th:main-intro}.
	 Therefore, Theorem \ref{th:main-intro} applies. Let $Q^*$ be given as therein. Computing derivatives as above, we have
\begin{align}
	R_f & \le  \frac{1}{\kappa^2}\sum_{1 \le i < j \le n}\E_{Q^*}[|\partial_{ij}f(X)|^2] = \frac{1}{\kappa^2}\sum_{1 \le i < j \le n}J_{ij}^2\E_{Q^*}\big[|K''(X_i-X_j)|^2\big]. \label{pf:QMGFbound2}
\end{align}
Using the assumption on $K''$, we find
\begin{align}
	\E_{Q^*}\big[|K''(X_i-X_j)|^2\big] &\le a\E_{Q^*}\big[e^{b|X_i-X_j|}\big] . \label{pf:QMGFbound3}
\end{align}
By assumption, $X_i-X_j$ has mean zero under $Q^*$.
It follows from the $\kappa$-log-concavity of  $Q^*$ that the law of $X_i-X_j$ is $(\kappa/2)$-log-concave (see, e.g.,  \cite[Theorem 3.7(a) and Theorem 3.8]{saumard2014log}).
This implies that it is subgaussian in the sense that
\begin{align*}
	\E_{Q^*}[e^{s(X_i-X_j)}] &\le e^{s^2/\kappa }, \qquad \forall s \in \R.
\end{align*}
Indeed, this can be deduced from the  log-Sobolev inequality (Theorem \ref{th:logsobolev}) via Herbst's argument or \cite[Theorem 1.3]{bobkov1999exponential}.
Thus, using \eqref{pf:QMGFbound3}, 
\begin{align*}
	\E_{Q^*}\big[|K''(X_i-X_j)|^2\big] &\le  a\E_{Q^*}\big[e^{b(X_i-X_j)}+e^{b(X_j-X_i)}\big] \le 2ae^{b^2/\kappa}.
\end{align*}
Combine this with \eqref{pf:QMGFbound2} to complete the proof.
\hfill \qed

\subsection{Doubly stochastic matrices}
We now turn to the proof of Theorem \ref{thm: regular graphs}. We first need a straightforward lemma about displacement convexity, which is likely known. 

\begin{lemma} \label{lem: entropy.convexity}
	Let $ Q_1, \dots, Q_n \in \P(\R)$ and $t_1, \dots, t_n \in [0, 1]$  be such that $\sum_{i=1}^n t_i = 1$. Then there exists a random vector $X = (X_1, \dots, X_n)$ such that $X_i \sim Q_i$ for each $i$ and 
	\begin{align*}
		H\bigg( \mathrm{Law} \bigg( \sum_{i=1}^n t_i X_i\bigg)\bigg) \le  \sum_{i=1}^n t_i H( Q_i).
	\end{align*}
\end{lemma}
\begin{proof}
	The proof is by induction on $n$, with the case $n=1$ holding trivially.
	Assume that the statement of the lemma is true for some $n$. Let $ Q_1, \dots, Q_{n+1} \in \P(\R)$ and $t_1, \dots, t_{n+1} \in [0, 1]$  be such that $\sum_{i=1}^{n+1} t_i = 1$. 
	Without loss of generality, assume that $t_{n+1} < 1$ and that $ Q_1, \dots, Q_{n+1}$ have finite entropy, as otherwise there is nothing to prove. 
	For $i=1,\dots, n$, define $\tilde{t}_i \coloneqq t_i / (1 - t_{n+1})$, so that $\sum_{i=1}^{n} \tilde{t}_i = 1$.  By assumption, we may find a random vector $( X_1, \dots, X_n)$ such that $X_i \sim Q_i$ for each $i=1,\ldots,n$ and 
	\begin{align} \label{eq: ind.hyp}
		H(\widetilde{Q}) \le  \sum_{i=1}^{n} \tilde{t}_i H(Q_i),
	\end{align}
	where  $\widetilde{Q}$ denotes the law of $\widetilde{X} :=  \sum_{i=1}^{n} \tilde{t}_i X_i$. By absolute continuity, there is a unique nondecreasing function $T : \R \to \R$ such that $\widetilde{Q} \circ T^{-1} = Q_{n+1}$. 
	The entropy functional is displacement convex \cite[Theorem 5.15(i)]{villani2003topics}, which means that the function 
	\begin{align*}
		[0,1] \ni t \mapsto H\big(\widetilde{Q} \circ (tT + (1-t)\mathrm{Id})^{-1}\big)
	\end{align*}
	is convex. In particular, letting $X_{n+1}=T(\widetilde{X})$, we find
	\begin{align*}
		H\big(\mathrm{Law}(t_{n+1}X_{n+1} + (1-t_{n+1})\widetilde{X})\big) &= H\big(\widetilde{Q} \circ (t_{n+1}T + (1-t_{n+1})\mathrm{Id})^{-1}\big) \\
		&\le t_{n+1}H(Q_{n+1}) + (1-t_{n+1})H(\widetilde{Q}).
	\end{align*}
	By \eqref{eq: ind.hyp} and the definition of $\tilde{t}_i$, we have $(1-t_{n+1})H(\widetilde{Q}) \le \sum_{i=1}^{n} t_i H(Q_i)$, completing the proof.
\end{proof}

\noindent{\textbf{Proof of Theorem \ref{thm: regular graphs}\eqref{regular graphs conv}.}}
Let us abbreviate
\begin{equation}
	M_n:=\sup_{Q \in \Pprod(\R^n)} M_n(Q), \label{eq: MF.Gibbs.sup}
\end{equation}
where we  define
\begin{align*} \begin{split}
		M_n(Q) &:=\int_{\R^n} f\,dQ - H(Q) \\
		&= \sum_{i=1}^n \int_{\R} V(x) \, Q_i(dx)  + \frac12 \sum_{i,j=1}^n J_{ij} \int_{\R} \int_{\R} K(x-y)\,Q_i(dx)Q_j(dy)  -  \sum_{i=1}^n H(Q_i), \end{split} 
\end{align*}
where the last equality used the symmetry of $J$ and $K$, the fact that the diagonal entries of $J$ are zero, and the tensorization of entropy. Recall that $\log \int_{\R^n} e^f\,dx = M_n + R_f$, by definition of $R_f$. 
We will complete the proof by showing that
\begin{align} \label{eq: MF.Gibbs.limit}
	M_n = n\sup_{Q \in \P(\R)} \left( \int_{\R} V\,dQ + \frac12 \int_{\R}\int_{\R} K(x-y)Q(dx)Q(dy) - H(Q)\right),
\end{align}
and that the optimizer $Q^*=Q^*_1 \times \cdots \times Q^*_n$ in \eqref{MFopt-mainthm} must be i.i.d., or $Q^*_1=\cdots=Q^*_n$.
Indeed, the i.i.d. form of $ Q^* $ implies $\E_{Q^*}[X_i-X_j]=0$ for all $i,j$. Using this and the assumption $\tr(J^2)=o(n)$, we may apply Corollary  \ref{cor:quadratic} to deduce that $R_f/n \to 0$, and Theorem \ref{thm: regular graphs}\eqref{regular graphs conv} follows.

The proof of the inequality $ (\ge) $ in \eqref{eq: MF.Gibbs.limit} is immediate upon restricting the supremum in \eqref{eq: MF.Gibbs.sup} to i.i.d.\ measures and using $\sum_{i,j=1}^nJ_{ij}=n$:
\begin{align*}
	M_n &\ge \sup_{Q\in\P(\R)}\bigg(  n \int_{\R} V(x) \, Q(dx) + \frac12 \sum_{i,j=1}^n J_{ij} \int_{\R} \int_{\R} K(x-y)\,Q(dx)Q(dy) - nH(Q)\bigg)\\
	& = n\sup_{Q \in \P(\R)} \bigg( \int_{\R} V\,dQ + \frac12 \int_{\R}\int_{\R} K(x-y)Q(dx)Q(dy) - H(Q)\bigg).
\end{align*}

To prove the inequality $ (\le) $ in \eqref{eq: MF.Gibbs.limit}, fix $Q = Q_1 \times \cdots \times Q_n \in \Pprod(\R^n)$ arbitrarily. By Lemma \ref{lem: entropy.convexity}, there exists a random vector $X = (X_1, \dots, X_n)$ such that $X_i \sim Q_i$ for all $i$ and 
\begin{equation}  \label{eq:Hineq} 
	H(\overline{Q}) \le  \frac 1n \sum_{i=1}^n  H(  Q_i ).
\end{equation}
where $\overline{Q}$ denotes the law of $\frac 1n \sum_{i=1}^n X_i$.
Using the concavity  $V$, we find
\begin{equation}
	\frac{1}{n}\sum_{i=1}^n \int_{\R} V(x) \, Q_i(dx) = \E\bigg[ \frac{1}{n}\sum_{i=1}^n  V(X_i) \bigg] \le \int_\R V\,d\overline{Q}. \label{eq:Vineq} 
\end{equation}
Let $Y=(Y_1,\ldots,Y_n)$ be an independent copy of $X$. 
Using the concavity of $K$ and the fact that $\sum_i J_{ij}=\sum_j J_{ij}=1$, we have
\begin{align}
	\frac{1}{n} \sum_{i,j=1}^n J_{ij} & \int_{\R} \int_{\R}  K(x-y)\,Q_i(dx)Q_j(dy) \nonumber\\
	&= \E\bigg[ \frac{1}{n} \sum_{i,j=1}^n J_{ij} K(X_i-Y_j) \bigg]  
	\le \E\bigg[ K \bigg( \frac{1}{n} \sum_{i,j=1}^n J_{ij}(X_i-Y_j)\bigg)\bigg] \nonumber \\
	&= \E\bigg[ K \bigg( \frac{1}{n} \sum_{i=1}^nX_i - \frac{1}{n} \sum_{j=1}^nY_j \bigg)\bigg]  = \int_\R \int_\R K(x-y) \,\overline{Q}(dx)\overline{Q}(dy). \label{eq:Kineq} 
\end{align}
Combining \eqref{eq:Hineq}, \eqref{eq:Vineq}, and \eqref{eq:Kineq}, we see that
\begin{align*}
	M_n(Q)/n &\le \int_\R V\,d\overline{Q} + \frac12\int_\R \int_\R K(x-y) \,\overline{Q}(dx)\overline{Q}(dy) - H(\overline{Q}) = M_n(\overline{Q}^{\otimes n})/n.
\end{align*}
In other words, for an arbitrary choice of product measure $Q$, we may increase $M_n(Q)$ by replacing $Q$ with the i.i.d.\ measure $\overline{Q}^{\otimes n}$. This completes the proof. \hfill \qed

\vskip.2cm

\noindent{\textbf{Proof of Theorem \ref{thm: regular graphs}\eqref{regular graphs weak law}.}}
We first justify the uniqueness claim. From part \eqref{main.unique} of Theorem \ref{th:main-intro}, we know that the optimizer $ Q^* \in \Pprod(\R^n) $ in \eqref{eq: MF.Gibbs.sup} is unique. It follows from the previous paragraph that this unique optimizer is in fact i.i.d., i.e., $ Q^* = Q^{\otimes n}$, where $Q \in \P(\R) $ is the (necessarily unique) optimizer of \eqref{eq: MF.Gibbs.limit}, which does not depend on $n$. This proves the desired uniqueness. 

Turning to the proof of \eqref{eq: regular graphs weak law}, recall that $R_f/n\to 0$, and use Corollary \ref{co:linearstat-concentration} and the aforementioned  i.i.d.\ form of the optimizer $ Q^\ast = Q^{\otimes n}$ to deduce that, for any 1-Lipschitz function $\varphi$,
\begin{align*}
\E_P\bigg[ \bigg(\frac{1}{n}\sum_{i=1}^n\varphi(X_i) - \int_\R \varphi\,dQ\bigg)^2\bigg]  \le \frac{(1 +  \sqrt{2R_f})^2}{\kappa n} \to 0, \quad \text{as } n \to \infty.
\end{align*}
This is enough to deduce that $\frac{1}{n}\sum_{i=1}^n\delta_{X_i}$ converges to $Q$ weakly in law. 
\hfill \qed

\subsection{Graphons proofs}

This section is devoted to the proof of Theorem \ref{thm:graphons}. 
For $W\in \mathcal{W}$ and any measurable function $\psi : \R^2 \to \R$ bounded from above, define  $T_{W,\psi} : \Punif \to \R$ by 
\begin{equation*}
T_{W,\psi}(\mu):= \E_{\mu^{\otimes 2}} \left[\psi(X_1,X_2) W(U_1,U_2)\right].
\end{equation*}
where $(U_1,X_1)$ and $(U_2,X_2)$ are independent with law $\mu$.
Note that $W \ge 0$ is integrable, so $T_{W,\psi}(\mu)$ is well-defined in $[-\infty,\infty)$. Let $\overline{\mu} := \mathrm{Unif}[0,1]\times \rho$,  and define $I : \Punif \to [0,\infty]$ by
\begin{equation*}
I(\mu):=H(\mu\,|\, \overline{\mu})=\int_0^1 H(\mu_u\,|\, \rho)\,du,
\end{equation*} 
with the second identity coming from the chain rule for relative entropy \cite[Theorem B.2.1]{dupuis2011weak}, and we recall that $\rho(dx)=e^{V(x)}dx$ is a probability measure. We begin with  two lemmas pertaining to the continuity of $T_{W,\psi}$.

\begin{lemma}\label{lem:properties}
	Let $ \mathcal{K} \subset \R$ be a compact interval.
	Let $\psi: \R^2 \to \R$ be supported on $ \mathcal{K}^2 $ and continuous when restricted to $ \mathcal{K}^2 $.
	\begin{enumerate}[(1)]
		\item If $\{W_\ell\}$ converges to $W$ in  strong cut metric and $W_\ell, W \ge 0 $, then 
		\begin{align*}
			\sup_{\mu\in \Punif}\Big|T_{W_\ell,\psi}(\mu)-T_{W,\psi}(\mu)\Big|\to 0.
		\end{align*}
		
		\item The map $\mu\to T_{W,\psi}(\mu)$ is continuous on $\{\mu\in \Punif:I(\mu)<\infty\}$, with respect to the topology of weak convergence.
	\end{enumerate}
\end{lemma}
\begin{proof}
	We begin with (1). Let $\mathcal{V}$ denote the space of functions $\phi:\mathcal{K}^2\mapsto \R$ of the form
	\begin{align}\label{eq:defU}
		\phi(x,y)=\sum_{i=1}^L c_i a_i(x)b_i(y),
	\end{align}
	for some $L\in \N$, $c_i\in \R$, and continuous functions $a_i,b_i:\mathcal{K} \to [0,1]$. It is easy to check that $\mathcal{V}$ is closed under multiplication, contains the constant functions, separates points in $\mathcal{K}^2$, and is a vector subspace of the space $C(\mathcal{K}^2)$ of continuous real-valued functions on $\mathcal{K}^2$.
	By the Stone-Weierstrass Theorem, we deduce that $\mathcal{V}$ is dense in $C(\mathcal{K}^2)$ with the supremum norm. Let $\varepsilon>0$, and find $\phi\in \mathcal{V}$ such that  $|\psi-\phi| < \varepsilon$ uniformly on $\mathcal{K}^2$.
	Extend the domain of $ \phi $ to $\R^2$ by setting $ \phi = 0 $ on the complement of $\mathcal{K}^2$. 
	 Then for all $\mu \in \Punif$,
	\begin{align*}
		\big|T_{W_\ell,\psi}(\mu)-T_{W_\ell,\phi}(\mu)\big|\le \varepsilon  \|W_\ell\|_{L_1[0,1]^2}, \quad \big|T_{W,\psi}(\mu)-T_{W,\phi}(\mu)\big|\le \varepsilon \|W\|_{L_1[0,1]^2}.
	\end{align*}
	Consequently, using the  triangle inequality, we have
	\begin{align}\label{eq:wtophi}
		\big|T_{W_\ell,\psi}(\mu)-T_{W,\psi}(\mu)\big|\le \varepsilon \|W_\ell\|_{L_1[0,1]^2} + \varepsilon \|W\|_{L_1[0,1]^2}  +\big|T_{W_\ell,\phi}(\mu)-T_{W,\phi}(\mu)\big|.
	\end{align}
	Since $\phi$ is of the form \eqref{eq:defU}, we have
	\begin{align*}
		T_{W_\ell,\phi}(\mu)=\sum_{i=1}^L c_i \int_{[0,1]^2}\bar a_i(u)\bar b_i(v)W_\ell(u,v) dudv,
	\end{align*}
	where we define $\bar a_i(u) := \E_{\mu}[a_i(X)\,|\,U=u]$, and $\bar b_i$ similarly. This yields
	\begin{align}\label{eq:phi}
		\big| T_{W_\ell,\phi}(\mu)- T_{W,\phi}(\mu)\big|\le \sum_{i=1}^L|c_i| d_\square(W_\ell,W).
	\end{align}
	Noting that  $ d_\square(W_\ell, W) \to 0$ implies $ \|W_\ell\|_{L_1[0,1]^2}  \to \|W\|_{L_1[0,1]^2} $, we may now combine \eqref{eq:wtophi} and \eqref{eq:phi}, sending $\ell\to\infty$ and then $\varepsilon \to 0$, to prove the claim (1).
	
	To prove (2), let $\mu_k$ be a sequence of measures in $\Punif$ converging weakly to $\mu_\infty$, such that $I(\mu_\infty)<\infty$. Let $W_\ell$ be a sequence of continuous functions in $\mathcal{W}$ converging in $L_1[0,1]^2$ to $W$. By the triangle inequality,
	\begin{align*}
		\big|T_{W,\psi}(\mu_k)-T_{W,\psi}(\mu_\infty)\big|\le 2\sup_{\nu\in \Punif} \big|T_{W,\psi}(\nu)-T_{W_\ell,\psi}(\nu)\big|+ \big|T_{W_\ell,\psi}(\mu_k)-T_{W_\ell,\psi}(\mu_\infty)\big|.
	\end{align*}
	The first term converges to $0$ as $\ell\to\infty$, by part (1) and the fact that convergence in $ L_1[0,1]^2 $ implies convergence in  strong cut metric. The second term converges to $0$ for fixed $\ell$ as $k\to\infty$, using the fact that $\mu_k$ converges weakly to $\mu_\infty$, and the set of discontinuity points of $W_\ell(\cdot,\cdot)\psi(\cdot,\cdot)$ is contained in $[0,1]^2\times \partial(\mathcal{K}^2)$, which has measure $0$ under $\mu_\infty^{\otimes 2}$ (as $\mu_\infty$ is absolutely continuous with respect to Lebesgue measure on $[0,1]\times \R$).
\end{proof}

\begin{lemma} \label{lem: TfW weak conv}
	Suppose $\mu_m$ is a sequence of measures in $\Punif$ converging weakly to $\mu_\infty$. Let $\psi:\R^2\mapsto\R$ be a continuous function, and let $W\in L_1[0,1]^2$. For $1\le m\le \infty$, let $(U^m_1,X^m_1),(U^m_2,X^m_2)\stackrel{i.i.d.}{\sim} \mu_m$. Then
	\begin{equation*}
		W(U^m_1,U^m_2) \psi(X^m_1, X^m_2) \stackrel{d}{\to} W(U^\infty_1,U^\infty_2) \psi(X^\infty_1,X^\infty_2). 
	\end{equation*}
\end{lemma}
\begin{proof}
	If $W$ is continuous, then the claim is immediate.
	For a general $W$, we proceed as follows:	
	Fix $\varepsilon>0$, and let $\mathcal{K}$ be a compact set such that $\PP(X^m_1\in \mathcal{K}, X^m_2\in \mathcal{K})\ge 1-\varepsilon$, which is again possible by tightness of $ \{\left(X^m_1, X^m_2\right)\}_{m \in \N} $. Let $g$ be a continuous function with  \[\|W-g\|_{L_1[0,1]^2}<\frac{\varepsilon}{1 \vee \sup_{x,y\in \mathcal{K}}|\psi(x,y)|}.\]  Then on the event $\left\{X^m_1\in \mathcal{K}, X^m_2\in \mathcal{K}\right\}$, we have
	\begin{equation*}
 \Big|W(U^m_1,U^m_2)\psi(X^m_1, X^m_2)-g(U^m_1,U^m_2)\psi(X^m_1, X^m_2) \Big|\le \varepsilon. 
	\end{equation*}
	Thus, for any continuous function $\phi:\mathbb{R} \to [0,1]$ which is $1$-Lipschitz, we have
	\begin{equation*}
		\Big|\E \phi\big(W(U^m_1,U^m_2)\psi(X^m_1, X^m_2)\big)-\E \phi\big(g(U^m_1,U^m_2)\psi(X^m_1, X^m_2)\big) \Big|\le 2\varepsilon.
	\end{equation*}
	Finally,
	\begin{equation*}
	\E \phi\big(g(U^m_1,U^m_2)\psi(X^m_1, X^m_2)\big)\to \E \phi(g(U^\infty_1,U^\infty_2)\psi(X^\infty_1, X^\infty_2)),
	\end{equation*}
	by the result for continuous functions. Thus
	\begin{equation*}
	\limsup_{m\to\infty}\Big|\E \phi(W(U^m_1,U^m_2)\psi(X^m_1, X^m_2))-\E \phi(W(U^\infty_1,U^\infty_2)\psi(X^\infty_1, X^\infty_2))\Big|\le 4\varepsilon.
	\end{equation*}
	Since $\varepsilon > 0$ is arbitrary, the proof of the lemma is complete.
\end{proof}

\noindent{\textbf{Proof of Theorem \ref{thm:graphons} \eqref{graphons conv}.}} 

We begin with some notation. For a measurable function $\psi : \R^2 \to \R$ which is bounded from above, define $M_n^\psi := \sup_{Q \in \Pprod(\R^n)}M_n^\psi(Q)$, where
\begin{align}  \begin{split}
		M_n^\psi(Q) &:= \sum_{i=1}^n \int_{\R} V(x) \, Q_i(dx)  +  \sum_{i,j=1}^n J_{ij} \int_{\R} \int_{\R} \psi(x,y)\,Q_i(dx)Q_j(dy)  -  \sum_{i=1}^n H(Q_i) \\
		&=  \sum_{i,j=1}^n J_{ij} \int_{\R} \int_{\R} \psi(x,y)\,Q_i(dx)Q_j(dy)  -  \sum_{i=1}^n H(Q_i\,|\,\rho).
	\end{split} \label{eq:MF.Gibbs-graphon}
\end{align}
Letting $\widetilde{K} : \R^2\to\R$ by $\widetilde{K}(x,y)=K(x-y)/2$, we are most interested in the choice $\psi=\widetilde{K}$, but treating a general $\psi$ will be helpful for a truncation argument.
Let $Q^*$ be as in Theorem \ref{th:main-intro}. 
With this notation, we have $\log \int_{\R^n} e^{f(x)}dx = M_n^{\widetilde{K}} + R_f$.
Corollary \ref{cor:quadratic} and the assumption that $\tr(J^2)=o(n)$ imply that $R_f/n \to 0$, and to prove Theorem \ref{thm:graphons} it will thus suffice to show that
\begin{align} \label{eq: MF.graphons.limit}
	\lim_{n \to \infty} M_n^{\psi}/n  =	\sup_{\mu\in \Punif} \big( T_{W,\psi}(\mu)-I(\mu)\big)
\end{align}
for any continuous function $\psi \le 0$.

	To this effect, use the assumption that $\{nJ\}_{n\ge 1}$ converges in weak cut metric to $W$ to conclude the existence of a sequence of permutations $\{\pi_n\}_{n\ge 1}$ with $\pi_n\in S_n$, such that $\{nJ^{(\pi_n)}\}_{n\ge 1}$ converges in strong cut metric to $W$, where  $J^{(\pi_n)}_{ij}:=J_{\pi_n(i)\pi_n(j)}$ for $1\le i,j\le n$. Since $\pi_n$ is a permutation, for any $Q =Q_1\times\cdots\times Q_n \in \Pprod(\R^n)$ we can write 
	\begin{align*}
		M_n^\psi(Q) 
		=&\sum_{i=1}^n \int_{\R} V(x)\tilde{Q}_{i}(dx)+\sum_{i,j=1}^nJ_{\pi_n(i)\pi_n(j)} \int_{\R}\int_{\R} \psi(x,y) \tilde{Q}_{i}(dx)\tilde{Q}_{j}(dy)-\sum_{i=1}^nH(\tilde{Q}_{i}),
	\end{align*}
	where $\tilde{Q}_i:=Q_{\pi_n(i)}\in \mathcal{P}(\R)$. Thus 
	\begin{equation*}
		\sup_{Q\in \Pprod(\R^n)}M_n^\psi(Q)=\sup_{\tilde{Q}\in \Pprod(\R^n)}\widetilde{M}_n^\psi(\tilde{Q}),
	\end{equation*}
	where $\widetilde{M}_n^\psi(\cdot)$ defined similarly to  $M_n^\psi(\cdot)$ in \eqref{eq:MF.Gibbs-graphon}, but with $J$ replaced by $J^{(\pi_n)}$. Since $nJ^{(\pi_n)}$ converges to $W$ in strong cut metric, by replacing $J$ with $J^{(\pi_n)}$ without loss of generality we assume throughout the rest of the proof that $nJ$ converges in strong cut metric to $W$.

To prove \eqref{eq: MF.graphons.limit}, we need the following construction which essentially embeds $\Pprod(\R^n)$ into $\Punif$ for all $n$.
For any $Q = Q_1\times\cdots\times Q_n\in \Pprod(\R^n)$, define a probability measure $\mu_n(Q)\in \Punif$ as follows: If $(U,X)\sim  \mu_n(Q)$, then $U\sim \mathrm{Unif}[0,1]$, and the conditional law of $X$ given $\{(i-1)/n < U \le i/n]\}$ is given by $Q_i$.  
Then we have
\begin{align*}
	T_{W_{nJ},\psi}(\mu_n(Q)) &= \frac{1}{n}\sum_{i,j=1}^n J_{ij} \int_\R \int_\R \psi(x,y) Q_i(dx) Q_j(dy),\\
	I(\mu_n(Q)) &= \frac{1}{n}\sum_{i=1}^nH(Q_i|\rho),
\end{align*}
and so 
\begin{equation}
	M_n^{\psi}(Q)/n=T_{W_{nJ},\psi}(\mu_n(Q))-I(\mu_n(Q)). \label{pf:graphon-Mn-identity}
\end{equation}

As a final preparation for the proof of \eqref{eq: MF.graphons.limit}, we argue that $\inf_n M_n^{\psi}/n > -\infty$. To see this, take $B \subset \R$ to be any compact set of positive $\rho$-measure, and define $\hat\rho \ll \rho$ by $d\hat\rho/d\rho = 1_B/\rho(B)$. Let $Q_i = \hat\rho$ for $i=1,\ldots,n$, and $Q=Q_1 \times \cdots \times Q_n$. Then
\begin{equation*}
	\frac{1}{n}\sum_{i=1}^n H(Q_i \,|\, \rho)= H(\hat\rho\,|\,\rho) = -\log \rho(B) < \infty,
\end{equation*}
and also 
\begin{equation*}
	T_{W_{nJ},\psi}(\mu_n(Q))  \ge - \|W_{nJ}\|_{L_1[0,1]^2}\sup_{x,y \in B}|\psi(x,y)|.
\end{equation*}
Since $\psi$ is continuous, it is bounded on the compact set $B$. Since $W_{nJ}$ converges in strong cut metric to $W$, we have $\|W_{nJ}\|_{L_1[0,1]^2} \to \|W\|_{L_1[0,1]^2}$, and thus the right-hand side is bounded. This proves that $\inf_n M_n^{\psi}/n > -\infty$.
We now prove the upper and lower bounds in \eqref{eq: MF.graphons.limit} separately.

\vskip.2cm

\noindent\textbf{Proof of the upper bound in \eqref{eq: MF.graphons.limit}}: 
Let $Q^n = Q^n_1 \times \cdots \times Q^n_n \in \Pprod(\R^n)$ be any near-optimizer of $M_n^\psi(\cdot)$, meaning
\begin{equation}
	M_n^{\psi}(Q^n) \ge M_n^\psi - o(n). \label{def:nearopt}
\end{equation}
Note that $ M_n^{\psi}(Q^n)/n$ is bounded from below by some constant $C$, as shown just above. Since $J$ has nonnegative entries and $\psi \le 0$, we have $T_{W_{nJ},\psi} \le 0$ which implies
\begin{align*}
	C \le M_n^{\psi}(Q^n)/n = T_{W_{nJ},\psi}(\mu_n(Q^n)) - I(\mu_n(Q^n)) \le - I(\mu_n(Q^n)).
\end{align*}
This implies $\sup_n I(\mu_n(Q^n)) < \infty$. Since the sub-level sets of $I$ are weakly compact, the sequence $(\mu_n(Q^n))$ has a limit point. Let $\mu_\infty$ be any limit point. Lower semicontinuity of $I(\cdot)$ gives $I(\mu_\infty)<\infty$. For each $m \in \N$, define $ \psi_m(x,y) \coloneqq \psi(x,y)1_{\{|x|,|y| \le m\}}$. Note that $\psi \le \psi_m \le 0$, and thus $T_{W_{nJ},\psi} \le T_{W_{nJ},\psi_m}$.  By part (1) of Lemma \ref{lem:properties}, 
\begin{align*}
	\sup_{\mu\in \Punif}\Big|T_{W_{nJ},\psi_m}(\mu)-T_{W,\psi_m}(\mu)\Big|\to 0,
\end{align*}
for all $ m \in \N $. Therefore, for all $ m$,
\begin{align*}
	\limsup_{n \to \infty} T_{W_{nJ}, \psi}(\mu_n(Q^n)) &\le \limsup_{n \to \infty} T_{W_{nJ},\psi_m}(\mu_n(Q^n)) \le \limsup_{n \to \infty} T_{W,\psi_m}(\mu_n(Q^n)) = T_{W,\psi_m}(\mu_\infty), 
\end{align*}
where the last step uses part (2) of Lemma \ref{lem:properties}.  The left-hand side above does not depend on $ m $, and thus
\begin{equation*}
	\limsup_{n \to \infty} T_{W_{nJ}, \psi}(\mu_n(Q^n)) \le \inf_{m \in \N} T_{W,\psi_m}(\mu_\infty) = T_{W,\psi}(\mu_\infty),
\end{equation*}
where the last equality follows from the monotone convergence theorem and the fact that $\psi_m \downarrow \psi$ pointwise.
Using the lower semicontinuity of $ I $, we deduce
\begin{align*}
	T_{W,\psi}(\mu_\infty) - I(\mu_\infty) 	&\ge \limsup_{n \to \infty} \big(T_{W_{n J},\psi} (\mu_n(Q^n)) - I(\mu_n(Q^n)) \big)   =\limsup_{n \to \infty} M_n^\psi (Q^n)/n.
\end{align*}
Bound the left-hand side by a supremum to prove the upper bound in \eqref{eq: MF.graphons.limit}. Moreover, once we prove \eqref{eq: MF.graphons.limit}, then this argument shows the following: for any near-optimizing sequence $Q^n = \Pprod(\R^n)$ in the sense of \eqref{def:nearopt}, the sequence $\{\mu_n(Q^n)\}$ is tight, and for any limit point $\mu_\infty$ of $\{\mu_n(Q^n)\}$ it holds that $\mu_\infty$ is an optimizer for the right-hand side of \eqref{eq: MF.graphons.limit}. 

{\ }

\noindent\textbf{Proof of the lower bound in \eqref{eq: MF.graphons.limit}}: 
To prove the lower bound in \eqref{eq: MF.graphons.limit}, we first claim that
\begin{align}
	\sup_{\mu\in \Punif} \big( T_{W,\psi}(\mu)-I(\mu)\big) = \sup_{\mu\in \Punif, \, \text{compact support}} \big( T_{W,\psi}(\mu)-I(\mu)\big). \label{pf:lowerbound-approx}
\end{align}
The inequality ($\ge$) is obvious.
To prove the reverse, let $\mu \in \Punif$ such that $I(\mu) < \infty$, and define $\mu^m \in \Punif$ with compact support by setting $d\mu^m/d\mu=1_{[0,1] \times [-m,m]}/\mu([0,1] \times [-m,m])$, which is well defined for large enough $m$. Then
\begin{align*}
	I(\mu_m) &= H(\mu^m\,|\,\overline{\mu}) = \int_{[0,1] \times \R} \log \frac{d\mu^m}{d\overline{\mu}} \,d\mu^m \\
	&=  \int_{[0,1] \times \R} \log \frac{d\mu^m}{d\mu} \,d\mu^m + \int_{[0,1] \times \R} \log \frac{d\mu}{d\overline{\mu}} \,d\mu^m.
\end{align*}
The second term converges to $I(\mu)$ by dominated convergence. The first term equals $-\log \mu([0,1] \times [-m,m])$ and vanishes as $m\to\infty$. Finally, since $W \ge 0$ and $\psi \le 0$, it is straightforward to check by monotone convergence that $T_{W,\psi}(\mu_m) \to T_{W,\psi}(\mu)$, and thus $T_{W,\psi}(\mu_m)-I(\mu_m) \to T_{W,\psi}(\mu)-I(\mu)$ as $m\to\infty$. This proves \eqref{pf:lowerbound-approx}.

Now, to prove the lower bound in \eqref{eq: MF.graphons.limit}, we let $\mu \in \Punif$ with compact support and with $I(\mu)<\infty$, and note that necessarily $\mu\ll  \overline{\mu}$. 
By defining $h(\cdot,\cdot):=\frac{d\mu}{d\overline{\mu}}$, we have $\int_{\R} h(u,\cdot)d\rho=1$ for a.e.\ $u \in [0,1]$ since both $\mu$ and $\overline{\mu}$ have uniform first marginal. For each $ i \in [n] $,  define $h^n_i : \R \to  [0,\infty)$ by 
\begin{equation*}
	h^n_i(x) \coloneqq n\int_{(i-1)/n}^{i/n} h(u,x) du.
\end{equation*}
By Fubini's theorem, $\int_{\R}h^n_i \, d\rho=1$ for all $i$.  We may thus define $Q^n = Q^n_1\times\cdots\times Q^n_n\in \Pprod(\R^n)$ by setting
$\frac{dQ^n_i}{d\rho} = h^n_i$, and define $\mu_n(Q^n)$ as before; note for later use the key identity $\frac{d\mu_n(Q^n)}{d\overline{\mu}}(u,x)=h_{\lceil nu\rceil}^n(x)$. If $\mathcal{K}$ denotes a compact interval such that $[0,1] \times \mathcal{K}$ contains the support of $\mu$, then $[0,1] \times \mathcal{K}$ also contains the support of $\mu_n(Q^n)$, and we may replace $\psi$ by $\psi 1_{\mathcal{K}^2}$ in the following argument.
Recalling the formula \eqref{pf:graphon-Mn-identity} for $M_n^\psi(Q)$, we may use  part (1) of Lemma \ref{lem:properties} to get
\begin{align} 
	T_{W,\psi}(\mu_n(Q^n))-\frac{1}{n}M_n^{\psi}(Q^n) -I(
	\mu_n(Q^n)) \to 0. \label{eq:graphon-LB1}
\end{align}
To complete the proof of the lower bound, we will show that
\begin{align}
	\lim_{n\to\infty} T_{W,\psi}(\mu_n(Q^n))= T_{W,\psi}(\mu), \quad \text{and} \quad I(\mu_n(Q^n))\le I(\mu), \ \forall n.	\label{eq:graphon-LB2}
\end{align}
Once \eqref{eq:graphon-LB2} is established, it will follow from the lower semicontinuity of $I$ that $I(\mu_n(Q^n))\to I(\mu)$, and we use \eqref{eq:graphon-LB1} to deduce
\begin{align*}
	\liminf_{n\to\infty}M_n^{\psi} / n \ge \lim_{n\to\infty}M_n^{\psi}(Q^n) / n = T_{W,\psi}(\mu) - I(\mu).
\end{align*}
This holds for every $\mu\in \Punif$ of compact support satisfying $I(\mu) < \infty$. Hence, taking the supremum and recalling \eqref{pf:lowerbound-approx} yields the desired lower bound in \eqref{eq: MF.graphons.limit}. 

It remains to prove \eqref{eq:graphon-LB2}. Note that
\begin{align*}
	\int_{[0,1] \times \R} \Big|\frac{d\mu}{d\overline{\mu}}-\frac{d\mu_n(Q^n)}{d\overline{\mu}}\Big|\,d\overline{\mu} &= \int_0^1 \int_\R |h(u,x) -  h^n_{\lceil nu\rceil}(x) |\,\rho(dx)du \\
	&= \E_{\overline{\mu}} \big| h(U,X)-\E_{\overline{\mu}}[ h(U,X)| \F_n]\big|,
\end{align*}
where $\F_n$ is the $\sigma$-field generated by $(\lceil nU\rceil, X)$. The right-hand side converges to $0$ by Levy's upwards convergence theorem, since $\E_{\overline{\mu}}|h(U,X)|=1<\infty$. Thus the probability measure $\mu_n(Q^n)$ converges in total variation to $\mu$, and the first claim in \eqref{eq:graphon-LB2} follows from part (2) of Lemma \ref{lem:properties}. To prove the second claim in \eqref{eq:graphon-LB2}, use convexity of $\varphi(x) \coloneqq x\log x$ for $x \ge 0$, along with Jensen's inequality, to get
\begin{align*}
	I(\mu_n(Q^n)) &= \E_{\overline{\mu}}\varphi( \E_{\overline{\mu}}[ h(U,X) | \F_n] ) \le \E_{\overline{\mu}}\varphi(  h(U,X)  ) = I(\mu).
\end{align*}
This proves \eqref{eq:graphon-LB2}, completing the proof of the lower bound, and thus Theorem \ref{thm:graphons}(1). \hfill \qed

\vskip.2cm

\noindent{\textbf{Proof of Theorem \ref{thm:graphons}\eqref{graphons weak law}.}}

We first discuss the optimization problem. The functional to be optimized can be written as
\begin{equation*}
	\Phi(\mu) := \frac12\int_{[0,1]\times \R}\int_{[0,1]\times \R} W(u,v)K(x-y)\mu(du,dx)\mu(dv,dy) - \int_0^1 H(\mu_u\,|\,\rho)\,du
\end{equation*}
We will show the existence of an optimizer via the weak upper semicontinuity: Since $W \ge 0$ and $K \le 0$, monotone convergence yields
\begin{align*}
2T_{W,\widetilde{K}}(\mu) =  \E_{\mu^{\otimes 2}}[W(U_1, U_2)K(X_1 - X_2)] = \inf_{m > 0} \E_{\mu^{\otimes 2}}\left[\big(W(U_1, U_2) K(X_1 - X_2) \big)\vee (-m)\right].
\end{align*}
For each $m$, the expectation appearing on the right-hand side is continuous as a function of $\mu \in \Punif$, by Lemma \ref{lem: TfW weak conv}. Hence, the left-hand side is upper semicontinuous. Since relative entropy is lower semicontinuous with compact sub-level sets, the existence of an optimizer follows.

We prove uniqueness of the optimizer via displacement convexity.  Letting $\widetilde{K}(x,y)=(1/2)K(x-y)$, we may rewrite $\Phi(\mu)=\Phi_1(\mu) + T_{W,\widetilde{K}}(\mu) - \Phi_2(\mu)$, where we define
\begin{equation*}
\Phi_1(\mu) := \int_0^1 \int_{\R} V(x)\,\mu_u(dx) du, \qquad  \Phi_2(\mu) := \int_0^1 H(\mu_u)\,du,
\end{equation*}
where we used the simple identity $H(\nu\,|\,\rho) = H(\nu) - \int_{\R^n} V\,d\nu$. 
Let $ \mu^0, \mu^1 \in \Punif$ be two optimizers, written in disintegrated form as $du\mu^i_u(dx)$ for $i=0,1$.
Let $F^i_u(x)=\mu^u_i(-\infty,x]$ denote the CDF, with generalized inverse $\overline{F}^i_u(y) := \inf\{x \in \R : y \le F^i_u(x)\}$. Then, for each $u \in [0,1]$, $T_u(x) := \overline{F}^1_u(F^0_u(x))$ denotes the unique nondecreasing function with $\mu^0_u \circ T_u^{-1} = \mu^1_u$. Since $F^i_u(x)$ is right-continuous in $x$ and measurable in $u$, it is jointly measurable in $(u,x)$, and the same is easily seen to be true for $\overline{F}^i_u(x)$ and thus $T_u(x)$. Consider the map $\overline{T} : [0,1] \times \R \to [0,1] \times \R$ given by $\overline{T}(u,x)=(u,T_u(x))$.
Define the interpolation $\mu^t := \mu^0 \circ ((1-t) \mathrm{Id} + t\overline{T})^{-1}$ for each $t \in [0,1]$. Then we have
\begin{align*}
	T_{W,\widetilde{K}} (\mu^t) &= \frac12 \int_{[0,1]\times \R}\int_{[0,1]\times \R} W(u,v) K(x-y) \mu^t (du,dx) \mu^t (dv,dy) \\
	&= \frac12 \int_{[0,1]\times \R}\int_{[0,1]\times \R} W(u,v) K\big((1-t)(x-y) + t(T_u(x)-T_u(y)) \big) \mu^0(du,dx) \mu^0(dv,dy).
\end{align*}
Since $ K $ is concave and $W \ge 0$, $t \mapsto T_{W,\widetilde{K}} (\mu^t)$ is concave. Note also that
\begin{align*}
	\Phi_2(\mu^t) = \int_0^1 H\big(\mu^0_u \circ ((1-t)\mathrm{Id} + tT_u)^{-1} \big)\,du
\end{align*}
is a convex function of $t$, by the  displacement convexity of entropy  \cite[Theorem 5.15(i)]{villani2003topics}. By the $\kappa$-concavity of $V$, the function $t \mapsto \Phi_1(\mu^t)$ is strictly concave, and we find that $t \mapsto \Phi(\mu^t)$ is strictly concave. Since $\mu^0$ and $\mu^1$ are both optimizers, we have $\Phi(\mu^0)=\Phi(\mu^1)$. Hence, we must have $\mu^0=\mu^1$, as otherwise the strict concavity would be contradicted.

With existence and uniqueness of the optimizer settled, we lastly prove the claim \eqref{eq: graphons weak law} in part \eqref{graphons weak law} of Theorem \ref{thm:graphons}.
Note that Theorem \ref{th:main-intro} implies uniqueness of the optimizer $Q^n$ in $\sup_{Q \in \Pprod(\R^n)}M_n(Q)$ for each $n$. 
Since $Q^n$ is optimal and thus a fortiori near-optimal, we may use the following fact proven in the course of proving the upper bound in Theorem \ref{thm:graphons}(1): The sequence $\{\mu_n(Q^n)\}$ is tight (since $Q^n$ is), and any limit point  is an optimizer for the right-hand side of \eqref{eq: MF.graphons.limit}. We have just shown the latter optimizer to be unique, and let us denote it $\mu^* \in \Punif$. Thus, $\mu_n(Q^n) \to \mu^*$ weakly.
From part (1) and Corollary \ref{co:linearstat-concentration}, for any bounded 1-Lipschitz function $\varphi : \R \to \R$   we have
\begin{equation*}
	\lim_{n \to \infty}\E_P\left[ \left(\frac{1}{n}\sum_{i=1}^n\varphi(X_i) - \frac{1}{n}\sum_{i=1}^n\E_{Q^n}[\varphi(X_i)]\right)^2\right] = 0.
\end{equation*}
Note that
\begin{equation*}
	\frac{1}{n}\sum_{i=1}^n\E_{Q^n}[\varphi(X_i)] = \frac{1}{n} \sum_{i=1}^n \int_{\R} \varphi(x) \,Q^n_i(dx) = \int_{[0,1] \times \R} \varphi(x) \,\mu_n(Q^n)(du, dx).
\end{equation*}
Using the weak convergence $\mu_n(Q^n) \to \mu^*$, the right-hand side converges to
\begin{equation*}
	\int_{[0,1] \times \R} \varphi(x) \mu^* (du, dx) = \int_\R \varphi \,dR^*, \ \ \text{where } R^* := \int_0^1 \mu^*_u\,du.
\end{equation*}
We deduce that $ \frac{1}{n} \sum_{i=1}^n \varphi(X_i) \to  \int_{\R} \varphi \,dR^*$ in probability for each bounded Lipschitz $\varphi$. This is enough to deduce the convergence in distribution $ \frac{1}{n} \sum_{i=1}^n \delta_{X_i} \to R^*$. \hfill \qed

\subsection{Proof of Lemma \ref{le:evenness}}
We first prove (1). When $f$ is even, we claim that (the density of) $Q^*$ is also even, which completes the proof because it implies $\E_{Q^*}[X_i]=0$ for all $i$.  To show that $Q^*$ is even, let $R_i(x):=Q^*_i(-x)$ for each $x \in \R$ and $i=1,\ldots,n$. Let $R=R_1\times \cdots \times R_n$. Then $\int_{\R^n} f\,dR=\int_{\R^n} f\,dQ$ by evenness of $f$, and clearly $H(Q)=H(R)$. Hence, $R$ is also an optimizer of \eqref{MFopt-mainthm}, and we deduce $R=Q^*$ by uniqueness of the optimizer.

We prove (2) by showing in this case that $Q^*_i=Q^*_j$ for all $i,j$. Suppose $f$ is invariant with respect to a transitive set $S$ of permutations of $[n]$. Fix $i,j \in \{1,\ldots,n\}$. Choose $\pi \in S$ such that $\pi(i)=j$, which is possible by the assumed transitivity of $S$. Let $R_k=Q^*_{\pi(k)}$ for each $k=1,\ldots,n$, and let $R = R_1 \times \cdots \times R_n$. The invariance of $f$ under $S$ ensures that $\int_{\R^n} f\,dR = \int_{\R^n} f\,dQ^*$.  Clearly, $H(R)=H(Q^*)$.  Hence, $R$ is also an optimizer of \eqref{MFopt-mainthm}, and we deduce that $R=Q^*$ by uniqueness. Since $\pi(i)=j$, this implies  $Q^*_i=R_i=Q^*_j$.

\section{Stochastic control proofs} \label{se:control-proofs}

As explained in Remark \ref{re:optimalcontrol}, the optimal admissible pair $(\alpha,X)$ for \eqref{control1} is given by
\begin{align}
	\alpha_g(t,x) &= \nabla_x \log \E[e^{n g(x+B_T-B_t)}], \label{Follmerdrift}
\end{align}
with $X=(X_t)_{t \in [0,T]}$ being the Brownian bridge with terminal law $P(dx)=Z^{-1}e^{ng(x)}\gamma_T(dx)$.
Letting $\PP$ denote the Wiener measure on $C([0,T];\R^n)$, the law $\QQ^P$ of this process $X$ can be characterized as the unique minimizer of $\QQ \mapsto H(\QQ\,|\,\PP)$ among $\QQ$ with time-$T$ marginal equal to $P$; see \cite[Proposition 6]{baudoin2002conditioned} or \cite[Lemma 10]{lehec2013representation}. This minimizer satisfies
\begin{align} \label{eq: entropy.minimizer}
	H(\QQ^P \, | \, \PP) =  H(P \, | \, \gamma_T) = \frac{1}{2}\E \left[\int_0^T |\alpha_g(t, X_t)|^2 \, dt\right].
\end{align}
Note that $H(P\,|\,\gamma_T) < \infty$, and so the pair $(\alpha_g,X)$ is admissible in the sense of Section \ref{se:control}.

\begin{proof}[Proof of Corollary \ref{co:stochcontrol}]
	Once the formulas \eqref{BBD} and \eqref{BBD-MF} are established,  the final claim follows immediately from Corollary \ref{co:refmeas}, applied with $V_i(x)=-x^2/(2T)$ for $i=1,\ldots,n$ and $\kappa=1/T$.

	To prove \eqref{BBD} and \eqref{BBD-MF}, we begin with the inequality ($\le$). Let $(\alpha,X)$ denote any admissible pair, and let $\QQ$ denote the law of $X=(X_t)_{t \in [0,T]}$.
	A well known argument using Girsanov's theorem \cite[Proposition 1]{lehec2013representation} yields
	\begin{align*}
		H(\QQ\,|\,\PP) &\le \frac12\E\int_0^T|\alpha(t,X_t)|^2\,dt = \frac12 \sum_{i=1}^n \E\int_0^T |\alpha_i(t,X_t)|^2dt.
	\end{align*}
	With $\QQ_T$ denoting the law of $X_T$, note that marginalizing (at time $T$) does not increase entropy: $H(\QQ\,|\,\PP) \ge H(\QQ_T\,|\,\gamma_T)$. Thus,
	\begin{align*}
		\E\Bigg[ g(X_T) -  \frac{1}{2n} \sum_{i=1}^n \int_0^T |\alpha_i(t,X_t)|^2dt\Bigg] &\le \int_{\R^n} g\,d\QQ_T - \frac{1}{n}H(\QQ_T\,|\,\gamma_T) \\
		&\le \sup_{Q \in \P(\R^n)} \left(\int_{\R^n} g\,dQ - \frac{1}{n}H(Q\,|\,\gamma_T)\right).
	\end{align*}
	Taking a supremum over all admissible pairs $(\alpha,X)$ proves the inequality ($\le$) in \eqref{BBD}.
	Now, if $(\alpha,X)$ is an \emph{distributed} admissible pair, then the same chain of inequalities holds, but also $\QQ_T$ is a product measure. We can thus deduce \eqref{BBD-MF} in the same manner.

		The inequality ($\ge$) in \eqref{BBD} and \eqref{BBD-MF} follows quickly from the entropy identity \eqref{eq: entropy.minimizer}. Starting with \eqref{BBD}, let $X=(X_t)_{t \in [0,T]}$ be the Brownian bridge with terminal law   $P(dx)=Z^{-1}e^{ng(x)}\gamma_T(dx)$. Let $\alpha_g$ be given as in \eqref{Follmerdrift}. By the Gibbs variational principle \cite[Proposition 1.4.2]{dupuis2011weak}, the supremum in \eqref{BBD} is attained by $Q=P$. Using $X_T \sim P$ and \eqref{eq: entropy.minimizer}, we obtain
		\begin{align*}
			\sup_{Q \in \P(\R^n)} \left(\int_{\R^n} g\,dQ - \frac{1}{n}H(Q\,|\,\gamma_T)\right) &= \int_{\R^n} g\,dP - \frac{1}{n}H(P\,|\,\gamma_T)  \\
			&= \E\bigg[ g(X_T) - \frac{1}{2n}\int_0^T |\alpha_g(t,X_t)|^2\,dt\bigg] \le V_{\mathrm{orig}}.
		\end{align*}
		This proves ($\ge$) in \eqref{BBD}, and also proves that $(\alpha_g,X)$ is optimal.
		Similarly, to prove the inequality ($\ge$) in \eqref{BBD-MF}, let $Q^* \in\Pprod(\R^n)$ be the unique optimizer in \eqref{BBD-MF}, which we know by Corollary \ref{co:refmeas} to take the form stated in Corollary \ref{co:stochcontrol}. Let $X=(X_t)_{t \in [0,T]}$ be the Brownian bridge with terminal law $Q^*$. Define $\alpha_{\widehat{g}}$ as in \eqref{Follmerdrift}, with $\widehat{g}(x)=\sum_{i=1}^n\E_{Q^*}[g(X)|X_i=x_i]$ in place of $g$. Using $X_T \sim Q^*$ and \eqref{eq: entropy.minimizer}, we obtain
		\begin{align*}
			\sup_{Q \in \Pprod(\R^n)} \left(\int_{\R^n} g\,dQ - \frac{1}{n}H(Q\,|\,\gamma_T)\right) &= \int_{\R^n} g\,dQ^* - \frac{1}{n}H(Q^*\,|\,\gamma_T)  \\
			&= \E\bigg[ g(X_T) - \frac{1}{2n}\int_0^T |\alpha_{\widehat{g}}(t,X_t)|^2\,dt\bigg] \le V_{\mathrm{distr}}.
		\end{align*}
		Indeed, note that $(\alpha_{\widehat{g}},X)$ is an admissible distributed pair because $Q^*$ is a product measure.
		This proves ($\ge$) in \eqref{BBD-MF}, and also proves that $(\alpha_{\widehat{g}},X)$ is optimal. 
\end{proof}

\subsection*{Acknowledgment}
We thank Ronen Eldan for helpful discussions and comments.
\bibliographystyle{amsplain}

\bibliography{biblio}
\end{document}